\documentclass[11pt]{amsart}
\usepackage{amsmath}
\usepackage{amsfonts}
\usepackage{amssymb}
\usepackage{graphicx}
\usepackage{hyperref}
\usepackage{mathrsfs}
\usepackage{pb-diagram}
\usepackage{epstopdf}
\usepackage{amsmath,amsfonts,amsthm,enumerate,amscd,latexsym,curves}
\usepackage{bbm}
\usepackage[mathscr]{eucal}
\usepackage{epsfig,epsf,epic}
\usepackage{caption}
\usepackage{subcaption}
\usepackage{multirow}
\usepackage{color}
\usepackage[all]{xy}
\usepackage{fourier}
\usepackage{tikz-cd}
\usepackage{calc}
\usepackage{float}
\usepackage{enumitem}
\usepackage{comment}

\newcommand{\hW}{\hat{W}}
\newcommand{\vol}{\mathrm{Vol}}

\newcommand{\cW}{\mathcal{W}}
\newcommand{\Fuk}{\mathcal{F}}
\newcommand{\Rad}{\mathrm{Rad}}

\newtheorem{thm}{Theorem}[section]
\newtheorem{lem}[thm]{Lemma}

\newtheorem{prop}[thm]{Proposition}
\newtheorem{definition}[thm]{Definition}

\newtheorem{remark}[thm]{Remark}

\newtheorem{assumption}[thm]{Assumption}

\numberwithin{equation}{section}

\begin{document}

\title[Categorical and topological entropies]{A comparison of categorical and topological entropies on Weinstein manifolds}

\author{Hanwool Bae}
\address[Hanwool Bae]{Center for Quantum Structures in Modules and Spaces, Seoul National University, Seoul, South Korea}
\email{hanwoolb@gmail.com}

\author{Sangjin Lee}
\address[Sangjin Lee]{Center for Geometry and Physics\\ Institute for Basic Science (IBS)\\ Pohang 37673, Korea}
\email{sangjinlee@kias.re.kr(Primary), sangjinlee@ibs.re.kr}

\begin{abstract}
	Let $W$ be a symplectic manifold, and let $\phi:W \to W$ be a symplectic automorphism. 
	This automorphism induces an auto-equivalence $\Phi$ defined on the Fukaya category of $W$.
	In this paper, we prove that the categorical entropy of $\Phi$ provides a lower bound for the topological entropy of $\phi$, where $W$ is a Weinstein manifold and $\phi$ is compactly supported. 
	Furthermore, motivated by \cite{Cineli-Erman-Ginzburg}, we propose a conjecture that generalizes the result of \cite{Newhouse, Przytycki, Yomdin}. 
\end{abstract}

\maketitle

\section{Introduction}
\label{section introduction}
\subsection{Introduction}
\label{subsection introduction}
Let \( W \) be a Weinstein manifold equipped with a compactly supported, exact symplectic automorphism \( \phi \). 
The pair \( (W, \phi) \) forms a discrete dynamical system. 
In this paper, we compare two invariants of this dynamical system. 

The first invariant is the {\em topological entropy}. 
This concept was introduced in the 1960s for compact spaces and extended to noncompact spaces in the 1970s. 
See \cite{Adler-Konheim-McAndrew, Hofer72, Hofer74}. 
For readers unfamiliar with dynamical systems, the books \cite{Hasselblatt-Katok95, Robinson98, Brin-Stuck02} are excellent references. 
Let \( h_{top}(\phi) \) denote the classical invariant associated with \( (W, \phi) \).

Recently, \cite{Dimitrov-Haiden-Katzarkov-Kontsevich} introduced the concept of {\em categorical entropy} for a pair \( (\mathcal{C}, \Phi) \), where \( \mathcal{C} \) is a triangulated category and \( \Phi: \mathcal{C} \to \mathcal{C} \) is an auto-equivalence. 
In our setting, the dynamical system \( (W, \phi) \) induces a categorical dynamical system \( (\mathcal{C}, \Phi) \) of symplectic-topological nature. 
Specifically, 
\begin{itemize}
	\item the (triangulated closure of the) wrapped Fukaya category \( \mathcal{W}(W) \) of a Weinstein manifold \( W \) is a triangulated category, and  
	\item the exact symplectic automorphism \( \phi \) induces an auto-equivalence \( \Phi: \mathcal{W}(W) \to \mathcal{W}(W) \).
\end{itemize}

This construction gives rise to the second invariant of the symplectic dynamical system \( (W, \phi) \): the categorical entropy of \( (\mathcal{W}(W), \Phi) \). Let \( h_{cat}(\phi) \) denote this invariant, which we refer to as the categorical entropy of \( \phi \). 

Since both entropies are invariants of the same dynamical system, it is natural to compare them. 
In this paper, we establish the inequality:  

\begin{gather}
	\label{eqn main inequality} 
	h_{cat}(\phi) \leq h_{top}(\phi).
\end{gather}

\begin{remark}
	\mbox{}
	\begin{enumerate}
		\item \cite{Kikuta-Ouchi, Mattei} studied the comparison of these two entropies in an algebro-geometric setting. 
		In particular, \cite{Kikuta-Ouchi} considered a pair \( (X, \phi) \), where \( X \) is a smooth projective variety, and \( \phi \) is a surjective endomorphism of \( X \). 
		Here, \( \phi \) induces an auto-equivalence \( \Phi \) on the derived category of coherent sheaves.
		The induced categorical dynamical system defines the categorical entropy of \( \phi \). 
		In this context, \cite{Kikuta-Ouchi} proved the equality \( h_{cat}(\Phi) = h_{top}(\phi) \). 
		\item In the symplectic-geometric setting, \cite{Bae-Choa-Jeong-Karabas-Lee} proved the inequality \( h_{cat}(\phi) \leq h_{top}(\phi) \) in specific cases.
	\end{enumerate}
\end{remark}

Before presenting formal statements of our results, we provide additional context, particularly references that explore the relationship between symplectic-topological invariants and the topological entropy of symplectic dynamical systems.
The symplectic-topological invariants we focus on here are {\em Floer-theoretic invariants}, as Lagrangian Floer theory plays a key role in connecting categorical and topological entropy in this work. 
However, there are numerous other significant results in the study of symplectic dynamical systems; see, for example, \cite{Smith12, Alves16, Dahinden18, Alves-Meiwes19, Dahinden20, Alves-Meiwes24}.

Floer theory produces various chain complexes and their homologies from symplectic manifolds and associated data. 
For instance, given a symplectic manifold \( X \) and a symplectic automorphism \( \phi: X \to X \), one can define the {\em fixed-point Floer homology} \( HF_*(\phi^n) \) for each \( n \in \mathbb{Z} \). 
When \( X \) is a compact surface and \( \phi \) is a symplectomorphism of \( X \), Fel'shtyn \cite{Felshtyn12} related the topological entropy of \( \phi \) to the growth rate of \( \dim HF_*(\phi^n) \) as \( n \to \infty \). 
Similarly, Frauenfelder and Schlenk \cite{Frauenfelder-Schlenk05, Frauenfelder-Schlenk15}, as well as Macarini and Schlenk \cite{Macarini-Schlenk11}, observed a connection between the growth rate of dimensions of {\em Lagrangian Floer homology} and the topological entropy of the corresponding dynamical system. 

In contact topology, a mathematical cousin of symplectic topology, it is known that the topological entropy of Reeb flows is bounded below by the exponential growth rate of the dimension of {\em contact homology}. 
For related results, see \cite{Alves-Colin-Honda19}.

We note that the results cited above primarily relate the topological entropy to the dimensions of Floer {\em homology}. More recently, leveraging ideas from {\em topological data analysis} and the theory of {\em persistence modules}, researchers have begun to study the relationship between the topological entropy of symplectic dynamical systems and the dimensions of Floer {\em chain complexes}. 
For example, Cineli, Ginzburg, G\"{u}rel \cite{Cineli-Erman-Ginzburg} explored the entropy of Hamiltonian flows using these methods.

In this paper, when \( W \) is a Weinstein manifold and \( \phi \) is a compactly supported symplectic automorphism, we relate the categorical entropy of \( \phi \) to the growth of dimensions of Lagrangian Floer homologies and the topological entropy of \( \phi \) to the growth of dimensions of Lagrangian Floer cochain complexes. 
Detailed statements of our results and an outline of the methods will be presented in the next subsection.

\subsection{Results}
\label{subsection results}
One of our main results is the inequality given in \eqref{eqn main inequality}.  
A key reason why we have an inequality instead of an equality in \eqref{eqn main inequality} is that \(h_{cat}(\phi)\) is invariant under compactly supported Hamiltonian isotopy, whereas \(h_{top}(\phi)\) is not.
In other words, if \(\phi_1\) and \(\phi_2\) are Hamiltonian isotopic, then \(h_{cat}(\phi_1) = h_{cat}(\phi_2)\). 
This is because \(\phi_1\) and \(\phi_2\) induce the same auto-equivalence on \(\cW(W)\). 
However, \(h_{top}(\phi_1)\) and \(h_{top}(\phi_2)\) do not necessarily coincide.  
Thus, one might expect the topological entropy to be more sensitive than the categorical entropy, making the inequality in Theorem \ref{thm main in introduction} plausible.  

\begin{thm}[=Theorem \ref{thm main}]  
	\label{thm main in introduction}  
	The categorical entropy of \(\phi\) bounds the topological entropy of \(\phi\) from below, i.e.,  
	\[
	h_{cat}(\phi) \leq h_{top}(\phi).
	\]  
\end{thm}  

\begin{proof}[Sketch of Proof]  
	Let \(\mathcal{C}\) be a triangulated category with a generator \(G\), and let \(\Phi: \mathcal{C} \to \mathcal{C}\) be an auto-equivalence.
	By \cite[Theorem 2.6]{Dimitrov-Haiden-Katzarkov-Kontsevich}, if \(\mathcal{C}\) is smooth and proper, then  
	\begin{gather}
		\label{eqn key theorem}
		h_{cat}(\Phi) = \lim_{n \to \infty} \dim \operatorname{Hom}\left(G, \Phi^n(G)\right).
	\end{gather}
	
	For the given dynamical system \((W, \phi)\), we note that the wrapped Fukaya category \(\cW(W)\) of \(W\) is smooth but not necessarily proper.
	Thus, \cite[Theorem 2.6]{Dimitrov-Haiden-Katzarkov-Kontsevich} cannot be applied directly.  
	
	However, a fully stopped partially wrapped Fukaya category is proper.
	Moreover, by Lemma \ref{lemma wrapped vs partially wrapped}, the categorical entropies of \(\phi\) on \(\cW(W)\) and on a partially wrapped Fukaya category are the same. 
	Thus, Equation \eqref{eqn key theorem} holds if \(\operatorname{Hom}\) refers to the morphism space in a fully stopped partially wrapped Fukaya category.  
	
	To complete the proof, we show that the right-hand side of Equation \eqref{eqn key theorem} bounds \(h_{top}(\phi)\) from below. 
	The key idea is to apply Crofton's inequality \eqref{eqn Crofton's inequality}, introduced in Lemma \ref{lemma Crofton's inequality}. 
	(Note that the inequality in \eqref{eqn Crofton's inequality} is a modification of \cite[Lemma 5.3]{Cineli-Erman-Ginzburg}. 
	For additional context, see Remark \ref{rmk Crofton inequality}.) 
	
	Specifically, let \(L_1 = G\) and \(L_2 = \phi^n(G)\), where \(G\) is a generator of a partially wrapped Fukaya category and \(n \in \mathbb{Z}_{\geq 1}\).  
	Then, combining \eqref{eqn key theorem} and \eqref{eqn proof 2}, we conclude that the exponential growth of the left-hand side of \eqref{eqn Crofton's inequality} bounds the categorical entropy of \(\phi\) from above, while the exponential growth of the right-hand side bounds the topological entropy of \(\phi\) from below.  
	See Lemma \ref{lemma Crofton's inequality} and Section \ref{section categorical vs topological entropy} for details.
\end{proof}

For the categorical entropy of \(\phi\), we primarily work with the wrapped Fukaya category \(\cW(W)\).  
However, there exists another triangulated category associated with \(W\): the compact Fukaya category \(\Fuk(W)\) (or its triangulated closure).  
This raises the question: why do we choose \(\cW(W)\) over \(\Fuk(W)\)?  

The reasons are as follows:  
\begin{itemize}  
	\item It is well-known that \(\cW(W)\) admits a Lagrangian generator.  
	For \(\Fuk(W)\), the existence of a Lagrangian generator is not known in general for arbitrary \(W\).  
	\item \(\cW(W)\) is a smooth category, while \(\Fuk(W)\) is not necessarily smooth.  
	Thus, \cite[Theorem 2.6]{Dimitrov-Haiden-Katzarkov-Kontsevich} cannot be applied to \(\Fuk(W)\).  
\end{itemize}

Nevertheless, if assumptions are added to resolve these issues, one could expect inequality \eqref{eqn main inequality} to hold for \(\Fuk(W)\) as well.  
With this perspective, we prove Theorem \ref{thm compact in introduction}.  

\begin{thm}[=Theorem \ref{thm main theorem for compact}]  
	\label{thm compact in introduction}  
	Let \((W, \phi: W \to W)\) satisfy the hypotheses in Lemma \ref{lem compact entropy}.  
	Let \(\Phi_{\Fuk(W)}\) denote the functor induced by \(\phi\) on the compact Fukaya category \(\Fuk(W)\).  
	Then the categorical entropy \(h_{cat}(\Phi_{\Fuk(W)})\) bounds the topological entropy of \(\phi\) from below:  
	\[h_{cat}(\Phi_{\Fuk(W)}) \leq h_{top}(\phi).\]  
\end{thm}

\subsection{Further questions}
\label{subsection further questions}
At the beginning of Section \ref{subsection results}, we highlighted a reason why we expect the inequality \eqref{eqn main inequality}: categorical entropy cannot distinguish members of a Hamiltonian isotopic class, whereas topological entropy can. 
Here, we present another philosophical reason for this expectation.  

To explain this reason, we review a property of topological entropy.  
By \cite{Newhouse, Przytycki, Yomdin}, it is known that  
\[
h_{top}(\phi) = \sup_{\text{compact submanifold } Y \subset W} \left(\text{the exponential growth rate of } \vol\left(\phi^n(Y)\right) \text{ with respect to } n\right).
\]  
In other words, \(h_{top}(\phi)\) can be computed by taking the supremum over all submanifolds \(Y\).  

On the other hand, the categorical entropy of \(\phi\) is only influenced by exact Lagrangian submanifolds; other submanifolds do not contribute to the categorical entropy.  

As a counterpart to the exponential growth rate of \(\vol\left(\phi^n(Y)\right)\), we define another entropy, called {\em barcode entropy}.
This notion is a slight modification of the {\em relative barcode entropy} introduced in \cite{Cineli-Erman-Ginzburg}.  
By definition, barcode entropy is not an invariant of the dynamical system \((W, \phi)\) but rather an invariant of \((W, \phi, L_1, L_2)\), where \(L_1\) and \(L_2\) are Lagrangian submanifolds of \(W\).  
We denote the barcode entropy for \((W, \phi, L_1, L_2)\) by \(h_{bar}(\phi; L_1, L_2)\).  
For details, see Section \ref{section barcode entropy}.  

With this definition, we establish Proposition \ref{prop barcode}:  
\begin{prop}[=Propositions \ref{prop bar vs top} and \ref{prop bar vs cat}]  
	\label{prop barcode}  
	For a pair of Lagrangians \((L_1, L_2)\) satisfying the conditions in Section \ref{section barcode entropy},  
	\[
	h_{cat}(\phi) \leq h_{bar}(\phi; L_1, L_2) \leq h_{top}(\phi).
	\]  
\end{prop}  

Based on Proposition \ref{prop barcode} and the preceding discussion, we ask whether the following equations hold:  
\begin{gather*}  
	h_{cat}(\phi) = \inf_{L_1, L_2} h_{bar}(\phi; L_1, L_2),\\  
	h_{top}(\phi) = \sup_{L_1, L_2} h_{bar}(\phi; L_1, L_2).  
\end{gather*}  

\subsection{Structure of the paper}
The paper consists of six sections except Section \ref{section introduction}.
Section \ref{section preparation} reviews definitions and preliminaries.
Sections \ref{section Crofton's inequality} and \ref{section categorical vs topological entropy} prove the main theorem, i.e., Theorem \ref{thm main in introduction}. 
Section \ref{section examples} discusses two examples: the first example shows that the inequality \eqref{eqn main inequality} can be strict, and the second example shows that categorical entropy could force topological entropy to be positive. 
Section \ref{section the case of compact Fukaya category} considers the compact Fukaya category of $W$ under some assumptions. 
Section \ref{section barcode entropy} is about the further questions described in Section \ref{subsection further questions}.

\subsection{Acknowledgment}
In Sections \ref{section Crofton's inequality} and \ref{section barcode entropy}, and the proof of Theorem \ref{thm main}, we use the idea given in \cite{Cineli-Erman-Ginzburg} heavily.
Also, Definition \ref{def barcode entropy} is originally introduced in \cite{Cineli-Erman-Ginzburg}. 
The second named author appreciates Viktor Ginzburg for explaining the key ideas of \cite{Cineli-Erman-Ginzburg} in a seminar talk and a personal conversation. 
The first named author is grateful to Otto van Koert for helpful comments. 

Hanwool Bae was supported by the National Research Foundation of Korea(NRF) grant funded by the Korea government(MSIT) (No.2020R1A5A1016126), and Sangjin Lee was supported by the Institute for Basic Science (IBS-R003-D1). 

\section{Preliminaries} 
\label{section preparation}
In this section, we give some preliminaries on Lagrangian Floer theory and (partially) wrapped Fukaya categories. Then we also introduce the notions of topological entropy and categorical entropy.

\subsection{Lagrangian Floer theory}
Let $\hW$ be a Weinstein manifold with a Liouville one form $\lambda$, which means that $(\hW,\omega=d\lambda)$ is a non-compact symplectic manifold such that the Liouville vector field $Z$, i.e., $Z$ is a vector field on $\hW$ characterized by $\iota_Z \omega = \lambda$, is gradient-like with respect to a certain
Morse function on $\hW$. See \cite{Cieliebak-Eliashberg12} for a reference on Weinstein manifolds.
Then, there exists a Weinstein domain $W$ whose completion is $\hW$. 
In other words, $(W, \omega|_{W}=d\lambda|_{W})$ is a compact symplectic submanifold with boundary of $\hW$ of codimension $0$, for which the Liouville vector field points outward along the boundary $\partial W$, and $\hW$ is obtained from $W$ by gluing the cylindrical end $[1,\infty) \times \partial W$ along $\partial W$, namely,
\[\hW= W \cup \left([1, \infty)\times \partial W \right).\]
Furthermore, on the cylindrical end $[1,\infty) \times \partial W$, the Liouville one form $\lambda$ satisfies
\[\lambda|_{[1,\infty) \times \partial W} := r \alpha\]
where $r$ is the coordinate for $[1,\infty)$ and $\alpha := \lambda|_{\partial W}$. We remark that $\partial W$ is a contact manifold with the contact form $\alpha$. Furthermore, since the Liouville vector field $Z$ does not vanish outside a compact subset of $\hW$, we may consider the orbit space $\partial_\infty \hW$ of the flow of $Z$ at infinity, which we will call the ideal boundary of $\hW$. One can deduce that $\partial W$ is naturally diffeomorphic to $\partial_\infty \hW$.

\begin{remark}
	We would like to point out that for a given Weinstein manifold $\hW$, the choice of a Weinstein domain $W$ is not unique.
	It seems that the choice of $W$ is related to the construction of wrapped Fukaya category of $\hW$.
	For example, the notion of {\em cylindrical at $\infty$} depends on the choice of $W$, see Definition \ref{def Lagrangian}. 
	However, it is well-known that the choice of $W$ does not change the Morita-equivalence class of the resulting wrapped Fukaya category of $\hW$. 
	In the current paper, in order to avoid unnecessarily confusion, we assume that a Weinstein domain $W$ is given and $\hW$ is defined to be the completion of $W$. 
\end{remark}

To discuss the Floer theory of a pair of Lagrangian submanifolds of $\hW$, let $J$ be an almost complex structure on $\hW$ that is compatible with the symplectic form $\omega = d \lambda$ and is of contact type at $\infty$.
The first condition means that the symplectic structure $\omega$ and the almost complex structure $J$ determine a Riemannian metric $g$ on $W$ given by
\[ g(\cdot ,\cdot ) = \omega( \cdot, J \cdot).\]
The latter condition means that the almost complex structure $J$ maps the Liouville vector field $Z$ to the Reeb vector field associated to $\alpha$ on $\partial W$ and preserves the contact distribution, which is necessary to apply maximum principles to $J$-holomorphic curves in $\hW$.

\begin{definition}
	\label{def Lagrangian}
	An exact Lagrangian submanifold $L$ of $(\hW,\omega)$ is {\em cylindrical (at $\infty$)} if $L$ and $\partial W$ intersect transversely and
	\[L \cap \left([1,\infty) \times \partial W \right) =  [1,\infty)\times \mathcal{L},\]
	where $\mathcal{L} := L \cap \partial W$. 
\end{definition}
For convenience, we will use the term ``Lagrangian'' instead of ``exact Lagrangian submanifold that is cylindrical at $\infty$''.

Let us assume that $L_1$ and $L_2$ are a transversal pair of Lagrangians in $\hW$.
Since $L_i$ is an exact Lagrangian submanifold, there is a primitive function 
\[h_i : L_i \to \mathbb{R},\]
such that $\lambda|_{L_i} = d f_i$. Let us fix such a primitive function $h_i$ for each Lagrangian $L_i$. 

Let $\mathcal{P}(L_1, L_2)$ be the space of paths from $L_1$ to $L_2$ in $\hW$. We define the action functional
\begin{equation}\label{eq action}
	\mathcal{A} = \mathcal{A}_{L_1, L_2} : \mathcal{P}(L_1, L_2) \to \mathbb{R}
\end{equation}
by
\[\mathcal{A}(\gamma) := -\int \gamma^*\lambda - h_1(\gamma(0)) + h_2(\gamma(1)), \forall \gamma \in \mathcal{P}(L_1,L_2).\]
Let us further equip the path space $\mathcal{P}(L_1,L_2)$ with the standard $L^2$-metric induced by $g$. Then it is straightforward to check that
\begin{itemize}
	\item critical points of the action functional $\mathcal{A}$ are  constant paths from $L_1$ to $L_2$, i.e., intersection points of $L_1$ and $L_2$, and
	\item gradient flows are given by strips 
	\[ u : \mathbb{R} \times [0,1] \to \hW\]
	satisfying 
	\[	u(s,0) \in L_1, u(s,1)\in L_2, \forall s\in \mathbb{R}\]
	and the $J$-holomorphic equation
	\begin{align*}
	\partial_s u+ J \partial_t u =0& , \forall (s,t) \in \mathbb{R} \times [0,1].
	\end{align*}
\end{itemize}

The Lagrangian Floer complex $CF^*(L_1,L_2)$ is given by the Morse complex for the action functional $\mathcal{A} =\mathcal{A}_{L_1,L_2}$ . Indeed, for a given field $k$,
\begin{itemize}
	\item $CF^*(L_1,L_2)$ is a graded $k$-vector space generated by intersection points $L_1 \cap L_2$, and
	\item the differential $\delta : CF^*(L_1,L_2) \to CF^{*+1} (L_1,L_2)$ is defined by counting $J$-holomorphic strips of index $1$ between two intersection points of $L_1$ and $L_2$.
\end{itemize}
We call the resulting cohomology Lagrangian Floer cohomology between $L_1$ and $L_2$ and denote it by $HF^*(L_1,L_2)$.

We would like to point out that the grading of the Floer cochain complex is not crucial in our remaining arguments. We will just assume that one of the followings holds.
\begin{itemize}
	\item either the Floer complex $CF^*(L_1,L_2)$ is $\mathbb{Z}/2$-graded, which is always possible, or
	\item it is $\mathbb{Z}$-graded assuming that 
	\[ 2 c_1(\hW) =0\]
	and  the Lagrangians $L_1$ and $L_2$ are graded in the sense of \cite[Section 12]{Seidel}. 
\end{itemize}
Later in Section \ref{section barcode entropy} where the grading is not important, we will usually denote the Floer cochain complex just by $CF(L_1,L_2)$ without the superscript $^*$.

We remark that one can extend the action functional $\mathcal{A}$ \eqref{eq action} as a function defined on $CF^*(L_1,L_2)$ as 
\begin{gather}
	\label{eqn non-archimedean norm}
	\mathcal{A}\left(\sum_{x_i \in L_1 \cap L_2} a_i x_i\right) = \max\left\{\mathcal{A}(x_i) | a_i \neq 0\right\}.
\end{gather}
We need the extended action functional $\mathcal{A}$ in Section \ref{section barcode entropy}.

\subsection{Partially wrapped Fukaya category}
Let $W$ be a Weinstein domain and let $(\hW,\omega=d\lambda)$ be the completion of $W$ as in the previous subsection. The purpose of this subsection is to provide a brief introduction of partially wrapped Fukaya category associated to a stop of $\hW$, mainly focusing on the definition of morphism space. There are two good references \cite{Sylvan19, Ganatra-Pardon-Shende18} for this topic. For convenience, we stick to the description given in \cite{Sylvan19}. A {\em stop} $\Lambda$ of $\hW$ is a (possibly empty) hypersurface with boundary of $\partial W$ (or equivalently $\partial_{\infty} \hW$) such that $\Lambda$ becomes a Liouville domain with the Liouville form $\lambda|_{\Lambda}$.

Let $\Lambda\subset \partial W= \partial_{\infty} \hW$ be a stop. The partially wrapped Fukaya category $\mathcal{W}(W,\Lambda)$ is an $A_{\infty}$-category associated to the pair $(W,\Lambda)$ whose objects are cylindrical Lagrangians $L$ of $\hW$ equipped with additional structures such as gradings, relative pin structures and local systems such that $L \cap \partial W \cap \Lambda =\emptyset$. 

To define morphism spaces of $\mathcal{W}(W,\Lambda)$, we consider quadratic Hamiltonians $H$ on $\hW$ that are compatible with the stop $\Lambda$ in the sense of \cite[Section 2.4]{Sylvan19}. Let us denote the set of all such Hamiltonians on $\hW$ by $\mathcal{H}(\hW)$. For any Hamiltonian $H$ on $\hW$, we denote the corresponding Hamiltonian vector field by $X_H$, that is, $X_H$ is a unique vector field on $\hW$ characterized by
\[\iota_{X_H}\omega = -dH.\]
We further consider an almost complex structure $J$ that is compatible with $\omega$ and is of contact type as in the previous subsection.

Let $L_1$ and $L_2$ be Lagrangians that are allowed to be objects of $\mathcal{W}(W,\Lambda)$. For a Hamiltonian $H \in \mathcal{H}(\hW)$, we define an action functional 
 \begin{equation}
 	\label{eqn action functional 2}
 \mathcal{A}_{L_1,L_2,H} : \mathcal{P}(L_1,L_2) \to \mathbb{R}
 \end{equation}
 by
 \[ \mathcal{A}_{L_1,L_2,H} (\gamma) = -\int \gamma^*\lambda + \int_{0}^{1} H(\gamma(t))dt- h_1(\gamma(0)) + h_2(\gamma(1)).\]
 
 As done in the previous subsection, it is easy to check that
 \begin{itemize}
 	\item critical points of the action functional $\mathcal{A}_{L_1,L_2,H}$ are $X_H$-chords from $L_1$ to $H_2$, where an $X_H$-chord from $L_1$ to $L_2$ means a path $\gamma:[0,1]\to \hW$ satisfying $\gamma(0)\in L_1$, $\gamma(1) \in L_2$ and $\dot{\gamma}(t) = X_H(\gamma(t))$, and
 	\item gradient flows are given by Floer strips 
 	\[ u : \mathbb{R} \times [0,1] \to \hW\]
 	satisfying 
 	\[	u(s,0) \in L_1, u(s,1)\in L_2, \forall s\in \mathbb{R}\]
	and the Floer equation
 	\[  \partial_s u+ J (\partial_t u - X_H(u)) =0 , \forall (s,t) \in \mathbb{R} \times [0,1].\]
 \end{itemize}
 Let us denote by $\mathcal{X}(L_1,L_2;H)$ the set of all $X_H$-chords from $L_1$ to $L_2$. There is a function $n_{\Lambda} : \mathcal{X}(L_1,L_2;H) \to  \mathbb{Z}_{\geq 0}$ defined by counting the intersection number of a given $X_H$-chord with the stop $\Lambda$. We take $\mathcal{X}_0(L_1,L_2;H)$ to be $n_{\Lambda}^{-1}(0)$.

 For a given field $k$, the partially wrapped Floer complex $CW_{\Lambda}^*(L_1,L_2)$ is defined as follows.
 \begin{itemize}
 	\item $CW^*_{\Lambda}(L_1,L_2)$ is a graded $k$-vector space generated by elements of $\mathcal{X}_0(L_1,L_2;H)$, and
 	\item the differential is defined by counting Floer strips of index $1$ between two elements of $\mathcal{X}_0(L_1,L_2;H)$.
 \end{itemize}
It can be shown as in \cite{Sylvan19} that this indeed defines a cochain complex. We call the resulting cohomology the partially wrapped Floer cohomology between $L_1$ and $L_2$ with respect to the stop $\Lambda$ and denote it by $HW^*_{\Lambda}(L_1,L_2)$.

There are two particularly interesting cases of a stop. The first one is the case $\Lambda=\emptyset$. In this case, the corresponding $CW^*_{\Lambda}(L_1,L_2)$ (respectively $HW^*_{\Lambda}(L_1,L_2)$) is called the wrapped Floer complex between $L_1$ and $L_2$ (resp. wrapped Floer cohomology between $L_1$ and $L_2$.) and denoted by $CW^*(L_1,L_2)$ (resp. $HW^*(L_1,L_2)$).

The second one is the case when the partially wrapped Floer complex (or more generally, partially wrapped Fukaya category) is fully stopped.
To be more precise, we define the notion of {\em full-stop}. 
\begin{definition}
	\label{def full stop}
	A stop $\Lambda$ is a {\em full-stop} if $HW^*_\Lambda(L_1,L_2)$ is of finite dimension for every pair $(L_1, L_2)$. 
\end{definition} 

We would like to note two things.
The first one is the reason why the notion of full-stop is interesting. 
The reason is actually mentioned in Definition \ref{def full stop}, i.e., the morphism spaces have finite dimensional homologies.
Equivalently, we can simply say that the corresponding partially wrapped Fukaya category is {\em proper}. 
This algebraic property plays a key role in the current article. 
Then, a question naturally arises, which is the second one we would like to note. 
The question asks the existence of full-stops. 
We note that for every Weinstein manifold, the existence is guaranteed.
For example, see \cite[Section 8.6]{Ganatra-Pardon-Shende18}.
%

To complete an introduction of partially wrapped Fukaya category, we need to explain the $A_{\infty}$-operations on it. Indeed, for a positive integer $k$, we call a disk with $(k+1)$-boundary punctures a $(k+1)$-gon. We further label the boundary punctures with $\{1,\dots,k+1 \ \mathrm{mod} \ k+2\}$ and call the boundary component between $i$-th and $(i+1)$-puncture the $i$th-boundary compoenent for $1\leq i \leq k+1$. For any collection $(L_1,\dots, L_{k+1})$ of Lagragians of $W$, the $A_\infty$-operation 
\[\mu^k_{\Lambda} : CW^*_{\Lambda}(L_k,L_{k+1})\otimes \dots \otimes CW^*_{\Lambda}(L_1,L_2) \to CW^*_{\Lambda}(L_1,L_{k+1})\] 
is defined by counting rigid Floer $(k+1)$-gons, which map the $i$-th boundary component to $L_i$ and are asymptotic to an $X_H$-chord from $L_i$ to $L_{i+1}$ belonging to $\mathcal{X}_0(L_i,L_{i+1};H)$ at the $i$-th puncture for $1\leq i \leq k$ and an $X_H$-chord from $L_1$ to $L_{k+1}$ at $k+1$-th pucnture.. Here $L_{k+2}$ is defined to be $L_1$ for simplicity. Especially the first operation $\mu^1_{\Lambda}$ coincides with the differential on the partially wrapped Floer complex up to sign. Then the operations $\mu^k_{\Lambda}$'s satisfy the $A_{\infty}$-relations as shown in \cite{Sylvan19}.

 The resulting $A_{\infty}$-category is the partially wrapped Fukaya category $\mathcal{W}(W,\Lambda)$ of the pair $(W,\Lambda)$. As done above, in the case $\Lambda= \emptyset$, the corresponding partially wrapped Fukaya category is called the wrapped Fukaya category of $W$ and  denoted by $\mathcal{W}(W)$. For future use in this paper, we further define the compact Fukaya category $\mathcal{F}(W)$ to be the full subcategory of $\mathcal{W}(W)$ generated by closed exact Lagrangian submanifolds of $\hW$.

\subsection{Topological entropy}
\label{subsection topological entropy}
Let $X$ be a topological space and let $\phi: X \to X$ be a continuous self-mapping defined on $X$. The notion of topological entropy $h_{top}(\phi)$ is defined in \cite{Adler-Konheim-McAndrew} for compact $X$ and in \cite{Hofer72,Hofer74} for non-compact $X$.

In particular, even though $X$ is non-compact, if a continuous self-mapping $\phi : X \to X$ is compactly-supported and so its support is contained in a compact subspace $X_0$ of $X$, then the following equality holds
	\[ h_{top} (\phi|_{X_0}) = h_{top}(\phi).\]
Since we only consider the topological entropy of compactly-supported smooth self-mappings on smooth manifolds in this paper, the above observation implies that it is enough to consider the topological entropy of a smooth self-mapping on a compact smooth manifold. Let $X$ be a compact smooth manifold (with or without boundary) of dimension $n$ equipped with a Riemannian metric $g$ in the rest of Section \ref{subsection topological entropy}. Then the following is a formulation of the topological entropy of a smooth map $\phi: X \to X$ due to Dinaburg \cite{Dinaburg70} and Bowen \cite{Bowen71}.



\begin{definition}[Dinaburg, Bowen]
	\label{def topological entropy}
	\mbox{} 
	\begin{enumerate}
		\item Let {\em $\Gamma_\phi^k$} denote the set of strings 
		\[\Gamma_{\phi}^k:=\Big\{\left(x, \phi(x), \dots, \phi^{k-1}(x)\right) \in X^k := X \times \dots \times X \text{  ($k$ factors).}\Big\}\]
		\item An {\em $\epsilon$-cubes} in $X^k$ is a product of balls in $X$ of radius $\epsilon$.
		\item For a subset $Y \subset X^k$, $\text{Cap}_{\epsilon} Y$ is the minimal number of $\epsilon$-cubes needed to cover $Y$. 
		\item The topological entropy of $\phi$, denoted by $h_{top}(\phi)$, is given by
		\[h_{top}(\phi) := \lim_{\epsilon \to 0} \limsup _{k \to \infty} \frac{1}{k} \log \mathrm{Cap}_{\epsilon} \Gamma_\phi^k.\] 
	\end{enumerate}
\end{definition}
Note that the $h_{top}(\phi)$ in Definition \ref{def topological entropy} (4) does not depend on a specific choice of a Riemannian metric $g$. See \cite{Dinaburg71, Bowen71} for more details on topological entropy.

We end this subsection by stating a property of topological entropy that plays a key role in the proof of Lemma \ref{lemma Crofton's inequality}.
For a $C^\infty$-submanifold $Y \subset X$ of dimension $m$, let 
\[\Gamma_{\phi|_Y}^k := \Big\{\left(y, \phi(y), \dots, \phi^{k-1}(y)\right) \in X^k | y \in Y \Big\}.\]
We note that the product metric on $X^k$ induces an $m$-dimensional volume form. 
Thus, we can measure the volumes of $\Gamma_{\phi|_Y}^k$ for all $k$.
The following is a significant observation due to Yomdin that the exponential growth rate of the volumes is a lower bound of $h_{top}(\phi)$.
\begin{thm}[Yomdin]
	\label{thm Yomdin thm}
	The topological entropy of $\phi$ is bounded by the exponential growth rate of the volume of $\Gamma_{\phi|_Y}^k$, i.e., 
	\[\limsup_{k \to \infty} \frac{1}{k} \log \vol(\Gamma_{\phi|_Y}^k) \leq h_{top}(\phi).\]
\end{thm}
See \cite{Yomdin, Gromov2} for the proof of Theorem \ref{thm Yomdin thm}.

\subsection{Categorical entropy}
\label{subsection categorical entropy}
In this subsection, we introduce the notion of categorical entropy. This requires preliminary knowledge on triangulated categories,  split-generators, etc. We recommend \cite{Gelfand-Manin13,Seidel} as a general reference on these.
We start by introducing the notions of complexity and categorical entropy, which are originally defined in \cite{Dimitrov-Haiden-Katzarkov-Kontsevich}.
\begin{definition}
	\label{def categorical entropy}
	Let $\mathcal{C}$ be a triangulated category with a split-generator $G$. 
	Let $\Phi$ be an auto-equivalence defined on $\mathcal{C}$.  
	\begin{enumerate}
		\item The {\em complexity of $E_2$ relative to $E_1$ at $t$} is a number in $[0,\infty]$ given by
		\[\delta_t(E_1;E_2):= \inf \Bigg\{ \sum_{i=1}^ke^{n_it} | \begin{tikzcd}
			0 \arrow{r} & A_1 \arrow{d}\arrow{r} & A_2 \arrow{d}\\
			& E_1[n_1] \arrow{ul} & E_1[n_2] \arrow{ul}
		\end{tikzcd} \dots \begin{tikzcd}
			A_{k-1} \arrow{r} & E_2\oplus E_2' \arrow[d]\\
			& E_1[n_k] \arrow{ul}
		\end{tikzcd}\Bigg\}.\] 
		\item For a given $t\in\mathbb{R}$, the {\em categorical entropy of $\Phi$ at $t$} is defined as
		\[h_{cat}(\Phi;t):=\lim_{n\to\infty}\frac{1}{n}\log\delta_t(G;\Phi^n(G))\in \{-\infty\}\cup \mathbb{R}.\]
	\end{enumerate}
\end{definition}

In the current paper, we only consider the case of $t=0$. 

\begin{definition}
	Let $\Phi: \mathcal{C} \to \mathcal{C}$ be an auto-equivalence defined on a triangulated category $\mathcal{C}$ with a generator $G$. 
	We define the {\em categorical entropy of $\Phi$} as 
	\[h_{cat}(\Phi):= h_{cat}(\Phi;0).\]
\end{definition}

\begin{remark}
	\label{rmk cateogrical entropy is non negative}
	We note that $h_{cat}(\Phi) = h_{cat}(\Phi;0) \geq 0$ by definition.
\end{remark}

Let $\mathcal{D}$ be a fully faithful subcategory of $\mathcal{C}$ such that 
\begin{itemize}
	\item $\mathcal{D}$ is a triangulated category, and
	\item the restriction of $\Phi$ to $\mathcal{D}$ defines an auto-equivalence on $\mathcal{D}$, i.e, $\Phi(\mathcal{D}) \subset \mathcal{D}$.  
\end{itemize}

It is known that there exists a localisation functor $l$ 
\[l : \mathcal{C} \to \mathcal{C}/\mathcal{D}.\] 
See \cite{Drinfeld}. 
Then, $\Phi$ induces an auto-equivalence defined on $\mathcal{C}/\mathcal{D}$ uniquely up to natural transformations.
More precisely, there exists a unique (up to natural transformation) dg functor
\[\Phi_{\mathcal{C}/\mathcal{D}}: \mathcal{C}/\mathcal{D}\to\mathcal{C}/\mathcal{D},\]
satisfying
\[\Phi_{\mathcal{C}/\mathcal{D}}\circ l=l\circ \Phi.\]

To be clear, let us use the following notations $\Phi_\mathcal{C}, \Phi_\mathcal{D}$, and $\Phi_{\mathcal{C}/\mathcal{D}}$,
\[\Phi_\mathcal{C}:= \Phi:\mathcal{C} \to \mathcal{C}, \Phi_\mathcal{D}:= \Phi_\mathcal{C}|_\mathcal{D}: \mathcal{D} \to \mathcal{D}, \text{ and } \Phi_{\mathcal{C}/\mathcal{D}}: \mathcal{C}/\mathcal{D} \to \mathcal{C}/\mathcal{D}.\]
Then, \cite[Theorem 3.8]{Bae-Choa-Jeong-Karabas-Lee} compares the categorical entropies of $\Phi_\mathcal{C}, \Phi_\mathcal{D}$, and $\Phi_{\mathcal{C}/\mathcal{D}}$.

\begin{lem}[Theorem 3.10 of \cite{Bae-Choa-Jeong-Karabas-Lee}]
	\label{lemma comparision of categorical entropies}
	The categorical entropies of $\Phi_\mathcal{C}, \Phi_\mathcal{D}, \Phi_{\mathcal{C}/\mathcal{D}}$ satisfy
	\[h_{cat}(\Phi_{\mathcal{C}/\mathcal{D}}) \leq h_{cat}(\Phi_\mathcal{C}) \leq \max \{h_{cat}(\Phi_\mathcal{D}), h_{cat}(\Phi_{\mathcal{C}/\mathcal{D}})\}.\]
\end{lem}

Let $\hW$ be a Weinstein manifold. Then let $W \subset \hW$ be an associated Weinstein domain and let $\Lambda$ be a stop in $\partial W=\partial_\infty \hW$. Let $\phi: \hW \to \hW$ be a compactly supported exact symplectic automorphism and. Then $\phi$ induces functors $\Phi: \mathcal{W}(W) \to \mathcal{W}(W)$ and $\Phi_\Lambda: \mathcal{W}(W,\Lambda) \to \mathcal{W}(W,\Lambda)$. 
Thanks to Lemma \ref{lemma comparision of categorical entropies}, one can compare $h_{cat}(\Phi)$ and $h_{cat}(\Phi_\Lambda)$.

\begin{lem}[Theorem 1.2 of \cite{Bae-Choa-Jeong-Karabas-Lee}]
	\label{lemma wrapped vs partially wrapped} 
	The induced functors $\Phi$ and $\Phi_\Lambda$ have the same categorical entropy, i.e.,
	\[h_{cat}(\Phi) = h_{cat}(\Phi_\Lambda).\]
\end{lem}
\begin{proof}
	We note that
	\[\mathcal{W}(W) := \mathcal{W}(W,\Lambda)/\mathcal{D},\]
	where $\mathcal{D}$ means the full subcategory of $\mathcal{W}(W,\Lambda)$ generated by all linking disks. Here linking disks are certain cylindrical Lagrangian disks of $\hW$ associated to $\Lambda$, for which the intersection with the domain $W$ can be chosen to lie arbitrarily close to $\Lambda \subset \partial W$. See \cite[Section 5]{Ganatra-Pardon-Shende18} for a definition of linking disk.
	
	Since $\phi$ is compactly supported, the restriction of $\Phi$ on $\mathcal{D}$ is the identity functor. 
	Thus, the categorical entropy of $\Phi|_\mathcal{D}$ is zero. 
	
	We note that, as mentioned Remark \ref{rmk cateogrical entropy is non negative}, \[h_{cat}(\Phi), h_{cat}(\Phi_\Lambda) \geq 0.\]
	By applying Lemma \ref{lemma comparision of categorical entropies}, one has 
	\[0 \leq h_{cat}(\Phi_\Lambda) \leq h_{cat}(\Phi) \leq \max\{h_{cat}(\Phi_\Lambda), 0\} = h_{cat}(\Phi_\Lambda).\]
	This completes the proof.
\end{proof}

\begin{remark}
	In Section \ref{section introduction}, we used the notation $h_{cat}(\phi)$ to denote $h_{cat} (\Phi)$ where $\Phi$ is the induced auto-equivalence on the wrapped Fukaya category of $\hW$. 
	In the above, we have seen that $\phi$ induces an auto-equivalence $\Phi$ on the wrapped Fukaya category $\cW(W)$ and an auto-equivalence $\Phi_{\Lambda}$ on the partially wrapped Fukaya category $\cW(W,\Lambda)$.
	In order to avoid confusion, we let $h_{cat}(\Phi)$ (resp.\ $h_{cat}(\Phi_{\Lambda})$) denote the categorical entropy of $\Phi$ on $\cW(W)$ (resp.\ $\Phi_{\Lambda}$ on $\cW(W, \Lambda)$).
\end{remark}

\section{Crofton's inequality}
\label{section Crofton's inequality}
The goal of this section is to prove Lemma \ref{lemma Crofton's inequality} which plays a key role in the proof of Theorem \ref{thm main in introduction}. 
In order to prove Lemma \ref{lemma Crofton's inequality}, we construct a family of Lagrangian submanifolds satisfying some conditions in Lemma \ref{lemma Lagrangian tomograph}.
By using the family of Lagrangians, we prove Lemma \ref{lemma Crofton's inequality} in Section \ref{subsection Crofton's inequality}. 

\subsection{Lagrangian tomograph}
\label{subsection Lagrangian tomograph}
Let $(\hW,d\lambda)$ be a Weinstein manifold and let $W \subset \hW$ be a Weinstein domain in the rest of this section.
In many places of this paper, we consider pairs of Lagrangians in $\hW$ satisfying the following condition. 
\begin{definition}
	\label{def good pair}
	A pair of Lagrangians $(L_1, L_2)$ in $\hW$ is {\em good} if $L_1$ and $L_2$ are disjoint in the cylindrical part, i.e.,
	\[L_1 \cap L_2 \cap \left([1,\infty) \times \partial W \right) = \varnothing.\]
\end{definition} 

For a good pair of Lagrangians, we construct a {\em Lagrangian tomograph} in Lemma \ref{lemma Lagrangian tomograph}.
The original construction of Lagrangian tomograph is given in \cite[Section 5.2.3]{Cineli-Erman-Ginzburg}, and our construction is a slight modification of the original one.

\begin{lem}
	\label{lemma Lagrangian tomograph}
	Let $(L_1, L_2)$ be a good pair of Lagrangians in $\hW$. 
	Then, for any $\epsilon >0$ and sufficiently large $d \in \mathbb{N}$, there is a sufficiently small $\delta>0$ and a family of Lagrangians $\{L^s\}_{s \in B_\epsilon^d}$, where $B_\epsilon^d$ is a $d$-dimensional closed ball of radius $\delta$ in the Euclidean space, such that 
	\begin{enumerate}
		\item[(i)] $L_1$ and $L^s$ are Hamiltonian isotopic to each other for all $s \in B_\epsilon^d$, 
		\item[(ii)] $d_H(L_1, L^s) < \frac{\epsilon}{2}$ for all $s \in B^d_\epsilon$, and
		\item[(iii)] $L^s$ and $L_2$ intersect transversely for almost all $s \in B_\epsilon^d$. 
	\end{enumerate}
\end{lem}
We note that the radius of the ball $B_\epsilon^d$ in the above lemma is $\delta$, not $\epsilon$. 
The reason why we use the notation is explained in Remark \ref{rmk choice of the family}.

Before going further, we briefly review the notion of {\em Hofer norm $d_H$ of a Hamiltonian isotopy}, which appears in the condition (ii) of Lemma \ref{lemma Lagrangian tomograph}.
Let $\varphi$ be a compactly supported Hamiltonian isotopy. 
Then, the Hofer norm of $\varphi$ is defined as 
\[\lVert \varphi \rVert_{Hofer} := \inf_H \int_{S^1} (\max_M H_t - \min_M H_t)dt,\]
where the infimum is taken over all $1$-periodic, time-dependent Hamiltonian $H$ generating $\varphi$. 
Moreover, one can define the Hofer distance between two Hamiltonian isotopic Lagrangians $L$ and $L'$ as 
\[d_H(L,L'):= \inf\{\lVert \varphi \rVert_{Hofer} | \varphi(L)= L'\}.\]

\begin{proof}[Proof of Lemma \ref{lemma Lagrangian tomograph}]
	Since $(L_1,L_2)$ is a good pair, there is a submanifold $W_0 \subset W$ of codimension 0 whose closure is compact such that 
	\[L_1 \cap L_2 \subset \mathrm{Int}(W_0) \subset W_0 \subset \mathrm{Int}(W),\]
	where $\mathrm{Int}(W_0)$ and $\mathrm{Int}(W)$ denote the interiors of $W_0$ and $W$ respectively. 
	Then, we choose a collection of real-valued functions 
	\[\{g_1, \dots, g_d | g_i: L_1 \to \mathbb{R}\},\]
	satisfying 
	\begin{enumerate}
		\item[(A)] $g_i(x)=0$ if $x \in L_1 \setminus W$, and
		\item[(B)] for all $x \in L_1 \cap W_0$, the cotangent fiber $T^*_xL_1$ is generated by $\{dg_i(x) | i =1, \dots, d\}$. 
	\end{enumerate}
	For any $s = (s_1, \dots, s_d) \in \mathbb{R}^d$, We set 
	\begin{gather*}
		f_s: L_1 \to \mathbb{R},\\
		x \mapsto s_1 g_1(x) + \dots + s_d g_d(x).
	\end{gather*}
	
	We note that there is a small neighborhood $U$ of $L_1$ in $\hW$, which is symplectomorphic to a small disk cotangent bundle of $L_1$. We call $U$ a Weinstein neighborhood of $L_1$.
	Then for $s \in \mathbb{R}^d$ such that $\lVert s \rVert \ll 1$, one can assume that the graph of $df_s$ is embedded into $\hW$. 
	Let $L^s$ be the embedded image of the graph of $df_s$ in $\hW$. 
	By the construction of $L^s$, (i) holds obviously. 
	
	Then, one can observe that (ii) holds for all $s \in \mathbb{R}^d$ sufficiently close to the origin. 
	Let $\delta >0$ be a sufficiently small number such that if $\lVert s \rVert < \delta$, then (ii) holds. 	
	Let $B_\epsilon^d \subset \mathbb{R}^d$ be the ball of radius $\delta$ centered at the origin.
	
	In order to prove (iii), we show that the following map $\Psi$ is submersive at every point of $\Psi^{-1}(L_2)$: 
	\begin{gather}
		\label{eqn psi}
		\Psi: B^d_\epsilon \times L_1 \to \hW, \\
		\notag (s,x) \mapsto df_s(x).
	\end{gather}
	In other words, if $\Psi(s,x) \in L_2$ for some $(s,x) \in B^d_\epsilon \times L_1$, we show that 
		\[D\Psi_{(s,x)}: T_{(s,x)}\left(B^d_\epsilon \times L_1\right) \simeq T_s B^d_\epsilon \oplus T_x L_1 \to T_{\Psi(s,x)}\hW\]
		is surjective. Then it will follow that $L^s$ and $L_2$ intersect transversely for almost all $s\in B_\epsilon^d$, i.e., (iii) holds.
	
	We equip $\hW$ with a Riemannian metric $g$ compatible with the Weinstein neighborhood of $L_1$ in the following sense. Indeed, note that any Riemannian metric $g_{L_1}$ on $L_1$ determines a natural Riemannian metric on its contangent bundle $T^*L_1$. Then identifying the Weinstein neighborhood $U$ with a disk subbundle of $T^*L_1$, we require that the restriction of $g$ to $U$ coincides with the unique natural Riemmanian metric induced by the Riemannian metric $g_{L_1} := g|_{L_1}$ in the above sense. This is always possible.
	
	Now let 
	\begin{gather}
		\label{eqn minimum distance}
		\ell := \min\left\{ d(x,y)| x \in L_1 \cap \left(W \setminus \mathrm{Int}(W_0)\right), y \in L_2 \cap \left(W \setminus \mathrm{Int}(W_0)\right)\right\},
	\end{gather}
	where $d(x,y)$ is the distance function with respect to the Riemannian metric $g$.
	Since both $L_1\cap (W\setminus \mathrm{Int}(W_0))$ and $L_2\cap (W\setminus \mathrm{Int}(W_0))$ are compact and $L_1 \cap L_2 \subset \mathrm{Int}(W_0)$, $\ell$ is well-defined and positive.
	
	Consider the Riemannian metric on $L_1$ given by the restriction of $g$. Let us denote by $\lVert \cdot \rVert$ the corresponding norm on $T^*_xL_1$.
		If the radius $\delta$ is sufficiently small, then $\lVert df_s(x) \rVert < \ell$ for all $(s,x) \in B^d_\epsilon \times W$ since $g_i$'s are compactly supported. We assume that $\delta$ is sufficiently small in this sense in the rest of the proof. 
	
		Now assume that $\Psi(s,x) \in L_2$ for some $(s,x)\in B^{d}_\epsilon \times L_1$.
		If $\Psi(s,x) \in \hW \setminus W$, then $\Psi(s,x) \in L_1$ since $g_i$'s are assumed to be zero over $\hW \setminus W$ by (A). 
		This contradicts to the assumption $L_1 \cap L_2 \subset W_0 \subset W$. Otherwise, if $\Psi(s,x) \in W \setminus \mathrm{Int}(W_0)$, then $d(\Psi(s,x),x) = \lVert df_s(x) \rVert < \ell$, which follows from our choice of the Riemannian metric $g$. But, this contradicts to Equation \eqref{eqn minimum distance}.
		
		The above paragraph shows that if $\Psi(s,x) \in L_2$, then $\Psi(s,x) \in W_0$, i.e., $L^s \cap L_2 \subset W_0$. 
		By the assumption (B), this means that $\Psi$ is submersive at every point in $\Psi^{-1}(L_2)$.
\end{proof}

\begin{remark}
	\label{rmk choice of the family}
	We also note that the the radius $\delta$ of $B^d_\epsilon$ in Lemma \ref{lemma Lagrangian tomograph} is determined by $\ell, \epsilon$, and the collection $\{g_1, \dots, g_d\}$.
	Among these factors, we would like to emphasize the effect of $\epsilon$ since in Section \ref{section barcode entropy}, we will vary $\epsilon$ and observe the effect to define the notion of {\em barcode entropy}.
	It is the reason why we use the notation $B_\epsilon^d$. 
\end{remark}

\subsection{Crofton's inequality}
\label{subsection Crofton's inequality}
In Section \ref{subsection Crofton's inequality}, we prove Lemma \ref{lemma Crofton's inequality}, i.e., a Crofton type inequality, which plays a key role in the proof of Theorem \ref{thm main in introduction}.

In order to state Lemma \ref{lemma Crofton's inequality}, we need some preparation. 
For $s \in B^d_\epsilon$ such that $L^s \pitchfork L_2$, let 
\[N(s) := |L^s \cap L_2|.\]
Then, $N(s)$ is finite for almost all $s \in B^d_\epsilon$. 
Moreover, $N(s)$ is an integrable function on $B^d_\epsilon$.

Since $B^d_\epsilon \subset \mathbb{R}^d$, $B^d_\epsilon$ carries the standard Euclidean metric. 
Let $ds$ be the volume form on $B^d_\epsilon$ induced by the Euclidean metric.

Let 
\[E := \Psi^{-1}(W_0).\]
Then, let us fix a metric $g_E$ on $E$ such that the restriction of $D\Psi$ to the normals to $\Psi^{-1}(y), y \in W$ is an isometry.
Since $\Psi$ is a proper submersion, $\Psi$ is a locally trivial fibration by Ehresmann's fibration theorem \cite{Ehresmann}.
Thus, the existence of such a metric is guaranteed. 

Now, we state Lemma \ref{lemma Crofton's inequality}. 
\begin{lem}
	\label{lemma Crofton's inequality}
	The following inequality holds:
	\begin{equation}
		\label{eqn Crofton's inequality}
			\int_{B^d_\epsilon} N(s) ds \leq C \cdot \vol(L_2 \cap W),
	\end{equation}
	where $C$ is a constant depending only on $\Psi, ds$, the fixed metric $g$ on $\hW$, and the fixed metric $g_E$ on $E$. 
\end{lem}
\begin{proof}
	Let $\Sigma:= \Psi^{-1}(L_2 \cap W)$. 
	Then, by definition, for all $s \in B_{\epsilon}^d$ such that $L^s \pitchfork L_2$,	one has 
	\[|(s \times L_1) \cap \Sigma| = |L^s \cap L_2| = N(s).\]
	Note that in the proof of Lemma \ref{lemma Lagrangian tomograph}, we have
	\[L^s \cap L_2 \subset W_0 \subset W,\]
	by choosing a sufficiently small $B^d_\epsilon$.

	We recall that $B^d_\epsilon$ carries the Euclidean metric and $L_1$ also carries a metric $g|_{L_1}$. 
	Thus, $B^d_\epsilon \times L_1$ carries a product metric. 
	On $E$, the restriction of the product metric gives another metric that does not need to be the same as $g_E$. 
	
	Let $\pi : E \hookrightarrow B^d_\epsilon \times L_1 \to B^d_\epsilon$ be the projection to the first factor. 
	Then, if $\vol_1(\cdot)$ denotes the volume with respect to the product metric on $E$, one has 
	\begin{gather}
		\label{eqn vol1}
		\int_{B^d_\epsilon} N(s) ds = \int_{B^d_\epsilon} |(s \times L_1) \cap \Sigma |ds = \int_{\Sigma} \pi^*ds \leq \vol_1(\Sigma).
	\end{gather}

	Let $\vol(\cdot)$ (resp.\ $\vol_2(\cdot)$) denote the volume with respect to the fixed metric $g$ (resp.\ $g_E$) on $W$ (resp.\ $E$). 
	Then, by Fubini theorem, one has 
	\begin{gather}
		\label{eqn vol2}
		\vol_2(\Sigma) = \int_{L_2 \cap W} \vol_2 \left(\Psi^{-1}(y)\right) dy|_{L_2} \leq \max_{y \in \Psi(E)} \vol_2\left(\Psi^{-1}(y)\right) \cdot \vol(L_2 \cap W). 
	\end{gather}

	We note that since $E$ is compact, 
	\begin{gather}
		\label{eqn vol1 vs vol2}
		\vol_1(\Sigma) \leq C_0 \cdot \vol_2(\Sigma),
	\end{gather}	
	where $C_0$ is a constant depending only on $g_E$ and the product metric on $E$. 
	
	By combining Equations \eqref{eqn vol1} -- \eqref{eqn vol1 vs vol2}, one concludes that 
	\[\int_{B^d_\epsilon} N(s) ds \leq C \cdot \vol(L_2 \cap W),\]
	where $C$ is a constant depending only on $\Psi, ds, g$, and $g_E$.
\end{proof}
\begin{remark}
	\label{rmk Crofton inequality}
	We note that Lemma \ref{lemma Crofton's inequality} is a slight modification of \cite[Lemma 5.3]{Cineli-Erman-Ginzburg}. 
	The original Crofton's inequality is proven in \cite{Arnold90a, Arnold90b}. 
	The reader can find additional context on Crofton's inequality in \cite[Section 5.2.1]{Cineli-Erman-Ginzburg}, especially after \cite[Remark 5.4]{Cineli-Erman-Ginzburg}, and in references therein. 
\end{remark}

\section{Categorical vs topological entropy}
\label{section categorical vs topological entropy}
In this Section, we prove our main theorem comparing categorical and topological entropy. 
To be more precise, let $(\hW,d\lambda)$ be a Weinstein manifold and let $W\subset \hW$ be an associated Weinstein domain. Let $\phi:\hW \to \hW$ be a compactly supported exact symplectic automorphism of $\hW$. 
Let $\Phi:\cW(W) \to \cW(W)$ denote the functor induced by $\phi$.
Then, we prove Theorem \ref{thm main}.

\begin{thm}[=Theorem \ref{thm main in introduction}]
	\label{thm main}
	The categorical entropy of $\Phi$ bounds the topological entropy of $\phi$ from below, i.e.,
	\[h_{cat}(\Phi) \leq h_{top}(\phi).\]
\end{thm}
\begin{proof}
In order to prove Theorem \ref{thm main}, we recall that every Weinstein manifold $\hW$ admits a Lefschetz fibration $\pi: \hW \to \mathbb{C}$ by \cite{Giroux-Pardon}.
Then, $\pi$ defines a Fukaya-Seidel category.
Moreover, it is known by \cite{Ganatra-Pardon-Shende18} that the corresponding Fukaya-Seidel category is the partially wrapped Fukaya category with the stop $\Lambda = \pi^{-1}(-\infty)$.   
Also, it is known that the {\em Lefschetz thimbles} of $\pi$ generate $\cW(W, \Lambda)$.  
Let $G$ denote the generating Lagrangian submanifold. 

We note that wrapping a Lagrangian $G$ means taking a time-$t$ flow of a Hamiltonian vector field $X_H$ for positive time $t$, where $H: \hW = W \cup ([1, \infty] \times \partial W) \to \mathbb{R}$ is a Hamiltonian function satisfying the following: 
\[H(x) = \begin{cases}
	0 \text{  if  } x \in W, \\
	r \text{  if  } x = (r, p) \in [1+\epsilon_0,\infty] \times \partial W \text{  with  } 0 < \epsilon_0 \ll 1.
\end{cases}
\]
Since $\cW(W,\Lambda)$ is fully stopped, there exists a positive number $t_0$ such that the time-$t_0$ flow of $X_H$, denoted by $\varphi_0$, satisfies that
\begin{enumerate}
	\item[(A)] $\left(\varphi_0(G),\phi^n(G)\right)$ is a good pair for all $n \in \mathbb{N}$, and
	\item[(B)] $HW_\Lambda\left(G,\phi^n(G)\right) = HF\left(\varphi_0(G),\phi^n(G)\right)$ for all $n \in \mathbb{N}$.
\end{enumerate}
See Figure \ref{figure wrapping_example}. 

\begin{figure}[h]
	\centering
\begingroup%
  \makeatletter%
  \providecommand\color[2][]{%
    \errmessage{(Inkscape) Color is used for the text in Inkscape, but the package 'color.sty' is not loaded}%
    \renewcommand\color[2][]{}%
  }%
  \providecommand\transparent[1]{%
    \errmessage{(Inkscape) Transparency is used (non-zero) for the text in Inkscape, but the package 'transparent.sty' is not loaded}%
    \renewcommand\transparent[1]{}%
  }%
  \providecommand\rotatebox[2]{#2}%
  \newcommand*\fsize{\dimexpr\f@size pt\relax}%
  \newcommand*\lineheight[1]{\fontsize{\fsize}{#1\fsize}\selectfont}%
  \ifx\svgwidth\undefined%
    \setlength{\unitlength}{283.46456693bp}%
    \ifx\svgscale\undefined%
      \relax%
    \else%
      \setlength{\unitlength}{\unitlength * \real{\svgscale}}%
    \fi%
  \else%
    \setlength{\unitlength}{\svgwidth}%
  \fi%
  \global\let\svgwidth\undefined%
  \global\let\svgscale\undefined%
  \makeatother%
  \begin{picture}(1,0.87)%
    \lineheight{1}%
    \setlength\tabcolsep{0pt}%
    \put(0,0){\includegraphics[width=\unitlength,page=1]{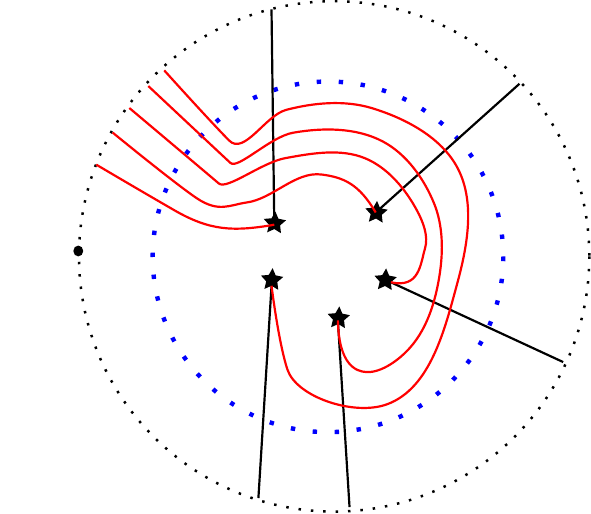}}%
    \put(0.05373871,0.43417442){\color[rgb]{0,0,0}\makebox(0,0)[lt]{\lineheight{1.25}\smash{\begin{tabular}[t]{l}$-\infty$\end{tabular}}}}%
  \end{picture}%
\endgroup%

	\caption{The interior of the black dotted circle is the base of a Lefschetz fibration $\pi$. The star marked points are the singular values and the black dot is $-\infty$. We note that the stop $\Lambda$ is given by $\Lambda = \pi^{-1}(-\infty)$. One can choose $G$ such that $\pi(G)$ is the union of all black curves.
	Similarly, $\pi\left(\varphi_0(G)\right)$ is the union of all red curves. 	Let $\pi(W)$ be contained in the interior of blue dotted circle. Then, $\left(\varphi_0(G),\phi^n(G)\right)$ is a good pair for all $n \in \mathbb{Z}$.	}
	\label{figure wrapping_example}
\end{figure}

For a given $n \in \mathbb{N}$, we apply Lemma \ref{lemma Lagrangian tomograph} for the good pair of Lagrangians $\left(\varphi_0(G),\phi^n(G)\right)$. 
Then, there exists a family of Lagrangian $\{L^s\}_{s \in B^d_\epsilon}$ such that 
\begin{enumerate}
	\item[(i)] $\varphi_0(G)$ and $L^s$ are Hamiltonian isotopic to each other for all $s \in B_\epsilon^d$, 
	\item[(ii)] $d_H(\varphi_0(G), L^s) < \frac{\epsilon}{2}$, and
	\item[(iii)] $L^s \pitchfork \phi^n(G)$ for almost all $s \in B_\epsilon^d$. 
\end{enumerate}

We note that one can find a family $\{L^s\}_{s \in B^d_\epsilon}$ which does not depend on $n$. 
To be more precise, we remark that in the proof of Lemma \ref{lemma Lagrangian tomograph}, the construction of $\{L^s\}_{s \in B^d_\epsilon}$ depends only on $\epsilon$, a collection of functions $\{g_1, \dots, g_d\}$, and $\ell$ in Equation \eqref{eqn minimum distance}. 
Since $\epsilon$ is a fixed, sufficiently small positive number, $\epsilon$ is independent of $n$. 
Similarly, $\{g_1, \dotsm g_d\}$ is a collection of functions not depending on $n$. 

We recall that in order to define the value $\ell$ in Equation \eqref{eqn minimum distance}, we had to fix $W_0 \subset W$ such that 
\[W_0 \subset \mathrm{Int}(W) \text{  and  } \varphi_0(G) \cap \phi^n(G) \subset \mathrm{Int}(W_0).\]
Without loss of generality, one can assume that $W_0$ not only satisfies the above two conditions, but also contains the support of $\phi$. 
Then, outside of $W_0$, $\phi^{n_1}(G)$ and $\phi^{n_2}(G)$ agree for all $n_i \in \mathbb{N}$. 
Thus, $\ell$ in Equation \eqref{eqn minimum distance} does not depend on $n$. 

Since we have a family $\{L^s\}_{s \in B^d_\epsilon}$ not depending on $n$, one can define the following function 
\[N_n(s) := |L^s \cap \phi^n(G)|.\]
We point out that for each $n \in \mathbb{N}$, $N_n(s)$ is an integrable function because of (iii).

By applying Lemma \ref{lemma Crofton's inequality}, we have 
\begin{gather}
	\label{eqn proof 1}
	\int_{B^d_\epsilon} N_n(s) ds \leq C \cdot \vol\left(\phi^n(G) \cap W\right).
\end{gather}
We note that the constant $C$ in \eqref{eqn proof 1} is independent of $n$. 

On the other hand, for $s \in B_{\epsilon}^d$ such that $L^s$ and $ \phi^n(G)$ intersect transversely, we have 
\begin{gather}
	\label{eqn proof 2}
	\mathrm{dim}HW_\Lambda\left(G,\phi^n(G)\right) = \mathrm{dim}HF\left(\varphi_0(G),\phi^n(G)\right)= \mathrm{dim}HF\left(L^s,\phi^n(G)\right) \leq N_n(s).
\end{gather}
The first equality holds because of (B), the second equality holds because of (i), and the last inequality holds because of the definition of Lagrangian Floer homology. 

Since \eqref{eqn proof 2} holds for almost every $s\in B_{\epsilon}^d$, by integrating Equation \eqref{eqn proof 2}, we have
\begin{gather}
	\label{eqn proof 3}
	\vol(B^d_\epsilon) \cdot \mathrm{dim}HW_\Lambda\left(G,\phi^n(G)\right) = \int_{B^d_\epsilon} \mathrm{dim}HW_\Lambda\left(G,\phi^n(G)\right) ds \leq \int_{B^d_\epsilon} N_n(s) ds.
\end{gather}

From two inequalities \eqref{eqn proof 1} and \eqref{eqn proof 3}, one has 
\begin{gather}
	\label{eqn proof 4}
	\vol(B^d_\epsilon) \cdot \mathrm{dim}HW_\Lambda\left(G,\phi^n(G)\right) \leq C \cdot \vol\left(\phi^n(G) \cap W\right).
\end{gather}
By taking $\limsup_{n \to \infty} \frac{1}{n} \log^+$ for the both hand sides of \eqref{eqn proof 4}, one has 
\begin{align}
	\label{eqn proof 5}
	h_{cat}\left(\Phi_\Lambda : \cW(W,\Lambda) \to \cW(W,\Lambda)\right) &= \limsup_{n\to\infty} \frac{1}{n} \log \dim HW_\Lambda\left(\varphi_0(G),\phi^n(G)\right) \\
	\notag &\leq \limsup_{n \to \infty} \frac{1}{n} \log \vol\left(\phi^n(G) \cap W\right).
\end{align}
The first equality in Equation \eqref{eqn proof 5} holds because of \cite[Theorem 2.6]{Dimitrov-Haiden-Katzarkov-Kontsevich} and because $\cW(W, \Lambda)$ is smooth and proper. 

Lemma \ref{lemma wrapped vs partially wrapped} says that
\[h_{cat}(\Phi) = h_{cat}(\Phi_\Lambda).\]
Also the following holds by Theorem \ref{thm Yomdin thm},
\begin{gather}
	\label{eqn proof 6}
	\limsup_{n \to \infty} \frac{1}{n} \log \vol\left(\phi^n(G) \cap W\right) \leq h_{top}(\phi).
\end{gather}
Then, \eqref{eqn proof 5} and \eqref{eqn proof 6} complete the proof.
\end{proof} 

\begin{remark}
	\label{rmk FS category}
	\mbox{}
	\begin{enumerate}
		\item In the proof of Theorem \ref{thm main}, we fix a Lefschetz fibration, and we use the corresponding Fukaya-Seidel category. 
		We note that if one fixes a {\em fully stopped} partially wrapped Fukaya category instead of a Fukaya-Seidel category, the same proof still works. 
		\item A direct application of Theorem \ref{thm main} is to show that the topological entropy of some symplectic automorphisms are positive.
		For example, \cite[Theorem 1.2]{Gong-Wang-Wue22} shows that symplectic automorphisms of a specific type, defined on the Milnor fibers of $A_n$-type, admit positive topological entropy. 
		For the symplectic automorphisms which \cite{Gong-Wang-Wue22} considers, one can check that their categorical entropy are positive by \cite[Themorem 7,14 and Lemma 8.5]{Bae-Choa-Jeong-Karabas-Lee}.
		Then, Theorem \ref{thm main} guarantees that the topological entropy is also positive. 
		We note that \cite{Gong-Wang-Wue22} shows their result by showing that the symplectic automorphisms have certain form of nonuniform hyperbolicity.
	\end{enumerate}
\end{remark}

\section{Examples}
\label{section examples}
In this section, we provide two examples.
The first example is a symplectic automorphism for which the inequality \eqref{eqn main inequality} is strict. 
The second example shows that categorical entropy can be strictly greater than the spectral radius of its induced map on the homology.

\subsection{The first example} 

Let $W$ be a 2-dimensional Weinstein domain such that $W \neq \mathbb{D}^2$. It is well-known that its wrapped Fukaya category is generated by the Lagrangian cocores, see \cite{Chantraine-Rizell-Ghiggini-Golovko17, Ganatra-Pardon-Shende18}. This ensures that the categorical entropy of an endo-functor on $\cW(W)$ is well-defined.

Let $U$ be a small open ball in $W$. It is well-known that there is a Hamiltonian diffeomorphism $\widetilde{\phi}: \overline{U} \to \overline{U}$ defined on the closure of $U$ such that
\begin{itemize}
	\item $\tilde{\phi}$ is the identity near the boundary of $\overline{U}$, and 
	\item $\tilde{\phi}$ has a positive topological entropy.
\end{itemize}  
Smale's horseshoe map is an example of such $\tilde{\phi}$. 

Since $\widetilde{\phi}$ is assumed to be the identity near the boundary of the closure of $U$, it admits a trivial extension to the whole Weinstein domain $W$, which we will call $\phi$. Since $\phi$ is a compactly supported Hamiltonian diffeomorphism, if we let $\Phi$ be its induced functor on $\mathcal{W}(W)$ as above, then we have
	\[h_{cat}(\Phi) = 0<h_{top}(\widetilde{\phi} = \phi|_{U}) \leq h_{top}(\phi).\]
	
Hence $\phi$ is an example showing that the inequality \eqref{eqn main inequality} can be strict.

\subsection{The second example}
Let $\phi:\hW \to \hW$ be a compactly supported exact symplectic automorphism of a Weinstein manifold $\hW$. Let $\phi_*: H_*(\hW) \to H_*(\hW)$ be the linear map on the homology of $\hW$ that $\phi$ induces.
We define the {\em spectral radius} of $\phi$ as the maximal absolute value of eigenvalues of $\phi_*$.  
Let $\Rad(\phi)$ denote the spectral radius of $\phi$. 
It is a well-known fact due to Yomdin \cite{Yomdin} that logarithm $\log \Rad(\phi)$ of the spectral radius is a lower bound of $h_{top}(\phi)$.
We refer the reader to \cite{Yomdin,Gromov2} for more details.
Since Theorem \ref{thm main} gives another lower bound of $h_{top}(\phi)$, i.e., $h_{cat}(\Phi)$, one can ask the relationship between two lower bounds of $h_{top}(\phi)$.
In this subsection, first, we give an example of $\phi$ such that 
\[\log \Rad(\phi) = 0 < h_{cat}(\phi).\]
The example shows that the categorical entropy could be an useful tool for proving positivity of topological entropy of symplectomorphisms. 
After introducing the example, we also introduce a construction of symplectomorphisms having positive topological entropy. 

To introduce the example, let $A$ and $B$ be $n$-dimensional spheres. 
Then let $\hW$ be the plumbing of the cotangent bundles $T^*A$ and $T^*B$ at a point.
In other words, $\hW$ is the Milnor fiber of $A_2$-type.

Let $\tau_A$ and $\tau_B$ be the Dehn twist defined on $W$ along $A$ and $B$, respectively.
Let us consider the symplectic automorphism on $\hW$ given by
	\[\phi = \tau_A \circ \tau_B^{-1}.\]
Now observe that the homology of $\hW$ is given by
	\[ H_*(\hW) = 
	\begin{cases} 
		\mathbb{Z} \langle [pt] \rangle &*=0, \\
		\mathbb{Z} \langle [A], [B]\rangle & *=n,\\
		0 &\text{otherwise.}
	\end{cases}\]

Since $\phi$ induces the trivial map on the zeroth homology $H_0(\hW)$, it is enough to consider its induced map on the $n$-dimensional homology $H_n(\hW)$ to compute $\Rad(\phi_*)$.

For that purpose, we consider the induced map of $\tau_A$ and $\tau_B$ on $H_n(\hW)$ separately.
\begin{enumerate}
	\item $(\tau_A)_*([A]) = (-1)^{n-1}[A]$.
	\item $(\tau_A)_*([B]) = [A]+[B]$.
	\item $(\tau_B)_*([A]) = [A] +(-1)^n [B]$.
	\item $(\tau_B)_*([B]) = (-1)^{n-1}[B]$. 
\end{enumerate}
In other words, $(\tau_A)_*$ and $(\tau_B)_*$ are represented by the matrices
	\[ (\tau_A)_* = 
	\begin{pmatrix} (-1)^{n-1} & 1 \\ 0 & 1\end{pmatrix} \text{  and  }
	(\tau_B)_* = 
	\begin{pmatrix} 1& 0 \\ (-1)^n & (-1)^{n-1}\end{pmatrix},
	\]
respectively.

Consequently, the map $\phi_* = (\tau_A)_* \circ (\tau_B^{-1})_*$ is represented by
\begin{align*}
	\begin{pmatrix} (-1)^{n-1} & 1 \\ 0 & 1 \end{pmatrix} \begin{pmatrix} 1& 0 \\ (-1)^n & (-1)^{n-1}\end{pmatrix}^{-1} &= \begin{pmatrix} (-1)^{n-1} & 1 \\ 0 & 1 \end{pmatrix} \begin{pmatrix} 1& 0 \\ 1 & (-1)^{n-1}\end{pmatrix}\\
	&= \begin{pmatrix}1+(-1)^{n-1}& (-1)^{n-1}\\ 1 & (-1)^{n-1} \end{pmatrix}.
\end{align*}

Let us now assume that $n$ is even. Then the above matrix is 
	\[
	\begin{pmatrix} 0 & -1\\ 1 & -1 \end{pmatrix}.
			\]
			
A straightforward computation shows that its eigenvalues are 
	\[ \frac{-1 + \sqrt{3} i}{2} \text{ and } \frac{-1- \sqrt{3}i}{2}.\]

Hence the spectral radius of $\phi_*$ is 
	\[ \left|\frac{-1+\sqrt{3}i}{2} \right| = \left|\frac{-1- \sqrt{3}i}{2} \right| =1.\]

On the other hand, \cite[Theorem 1.5]{Bae-Lee23} says that the categorical entropy of the induced map of $\phi$ on the wrapped Fukaya category is the same as its {\em stretching factor}.
Moreover, one can easily compute the stretching factor of $\phi$ by employing the techniques in \cite{Bae-Lee23}, and the stretching factor of the given example $\phi$ is  
	\[ \frac{3+\sqrt{5}}{2}.\] 
We note that the detailed computation could be found in \cite[Section 11.1]{Bae-Lee23}. 

Finally, we can prove that $\phi$ has positive topological entropy even though the spectral radius of $\phi$ is $1$, since 
	\[ \log \Rad( \phi) =0 <\frac{3+\sqrt{5}}{2} = h_{cat} (\Phi)  \leq h_{top}(\phi).\]
	
We would like to note that the above example $\phi$ is a symplectic automorphism constructed by a construction introduced in \cite{Lee24}. 
We introduce the construction given in \cite{Lee24} and prove that if a symplectic automorphism $\psi$ is compactly supported Hamiltonian isotopic to any symplectic automorphism $\phi$ constructed in \cite{Lee24}, then $\psi$ has positive topological entropy. 

The following is the construction given in \cite{Lee24} with a technical condition. 
\begin{definition}
	\label{def Penner type}
	Let $M$ be a symplectic manifold and $\phi:M \to M$ be a symplectic automorphism. 
	We call that $\phi$ is of {\em Penner type} if there exist two collections of Lagrangian spheres in $M$, 
	\[\{\alpha_1, \dots, \alpha_s\} \text{  and  } \{\beta_1, \dots, \beta_t\}\]
	satisfying the following conditions:
	\begin{enumerate}
		\item[(i)] $\phi$ is a product of positive powers of $\tau_i$ and negative powers of $\sigma_j$ where $\tau_i$ and $\sigma_j$ are generalized Dehn twists along $\alpha_i$ and $\beta_j$ respectively.
		\item[(ii)] $\{\alpha_1, \dots, \alpha_s\}$ and $\{\beta_1, \dots, \beta_t\}$ are pair-wisely disjoint collections of Lagrangian spheres and $\alpha_i$ and $\beta_j$ are transverse to each other.
		\item [(iii)] Let $G$ be a graph such that 
			\begin{itemize}
				\item whose vertex set is $\{\alpha_1, \dots, \alpha_s, \beta_1, \dots, \beta_t\}$ and
				\item two vertices $v, w \in \{\alpha_1, \dots, \alpha_s, \beta_1, \dots, \beta_t\}$ are connected by $k$ edges if the number of intersection points of two Lagrangian spheres is $k$, i.e., $|v \cap w| =k$.
			\end{itemize}
		Then, the graph $G$ is a tree.  
	\end{enumerate}
\end{definition}

\begin{remark}
	\label{rmk Penner type}
	We remark that Definition \ref{def Penner type} is motivated from Penner's construction of pseudo-Anosov surface mapping classes, introduced in \cite{Penner88}. 
	One can easily see that two conditions (i) and (ii) of Definition \ref{def Penner type} are generalizations of the original construction of Penner. 
	The third condition (iii) is a technical condition which we need, in order to employ the result of \cite{Bae-Lee23}. 
	However, we expect that the technical condition could be removed.
	More details will be mentioned in Remark \ref{rmk technical condition}.
\end{remark}

Now, we prove the main result of the present subsection.
\begin{thm}
	\label{thm Penner type}
	Let $M$ be a symplectic manifold of dimension $\geq 6$. 
	If a symplectic automorphism $\phi:M \to M$ is of Penner type, then
	\[h_{top}(\phi) > 0.\]
\end{thm}
\begin{proof}
	We note that any Penner type $\phi$ is compactly supported by Definition \ref{def Penner type}.
	Moreover, the compact support of $\phi$ is given as a small neighborhood of 
	\[\cup \alpha_i \bigcup \cup \beta_j.\]
	Thanks to Weinstein's neighborhood theorem \cite{Weinstein71}, we could assume that the compact support is a subset of a plumbing space $P$ of $T^*S^n$ where $n$ is the dimension of the Lagrangian spheres.
	Moreover, the plumbing pattern of $P$ is determined by the intersection pattern of $\alpha_i$ and $\beta_j$.
	Thus, the plumbing pattern is the graph $G$ defined in \ref{def Penner type} (iii). 
	
	Simply, we have 
	\[\text{compact support of  } \phi \subset P.\]
	Now, we can consider a symplectic automorphism $\phi_0: P \to P$, which is a natural extension of the restriction of $\phi$ on the compact support.
	Then, we have that the topological entropy of $\phi$ and $\phi_0$ coincide. 
	Thus, it is enough to show that $h_{cat}(\Phi_0) >0$. 
	
	Note that the domain of $\phi_0$ is a plumbing space $P$ whose plumbing pattern is a tree. 
	Thus, we can employ the results of \cite{Bae-Lee23}. 
	Finally, \cite[Theorems 1.5, 7.1]{Bae-Lee23} proves that $h_{cat}(\Phi_0)>0$.   
\end{proof}

\begin{remark}
	\label{rmk technical condition}
	We end this subsection with remarks on Theorem \ref{thm Penner type} and Definition \ref{def Penner type} (iii). 
	\begin{enumerate}
		\item We note that the proof of Theorem \ref{thm Penner type} used the results in \cite{Bae-Lee23} without details. 
		We want to simply explain the idea of \cite{Bae-Lee23} before moving on to the next section. 
		The main result of \cite{Bae-Lee23} is to show that if $\phi$ is of Penner type, then $\phi$ induces a functor on Fukaya category having {\em simple asymptotic behavior}. 
		In other words, $\Phi^n$ has a simple, in a categorical sense, behavior as $n \to \infty$.
		It is expected by \cite{Lee24} that shows a simple, in a geometric sense, asymptotic behavior of $\phi$.
		To prove the simple asymptotic behaviors of Penner type symplectic automorphisms, \cite{Lee24, Bae-Lee23} generalized the notion of {\em measured train track} and linear algebras on it in geometric and categorical senses. 
		For more details on measured train tracks, we refer the reader to \cite{Farb-Margalit12}
		\item As mentioned in Remark \ref{rmk Penner type}, Definition \ref{def Penner type} (iii) is a technical condition. 
		We note that thanks to the technical condition, we can guarantee that the plumbing pattern of $P$ in the proof of Theorem \ref{thm Penner type} is a tree. 
		Then, since there exist known generators of the compact Fukaya category of $P$, we can employ the techniques in \cite{Bae-Lee23}. 
		
		However, we expect that one can choose specific generators of compact Fukaya category of any plumbing space. 
		This is an ongoing project of the second-named author, together with Wonbo Jeong and Dogancan Karabas.
		After fixing generators of compact Fukaya category of a general plumbing space, we expect that the result of \cite{Bae-Lee23} could be generalized on a general plumbing space. 
		And, it allows us to drop the technical condition, i.e., Definition \ref{def Penner type} (iii).
	\end{enumerate} 
\end{remark}

\section{The case of compact Fukaya category}
\label{section the case of compact Fukaya category}
As mentioned in the introduction, we prove that a variant of Theorem \ref{thm main} holds for compact Fukaya category under an additional assumption. 
Let $(\hW,d\lambda)$ be a Weinstein manifold and let $W\subset \hW$ be an associated Weinstein domain.
The assumption we consider is a kind of ``{\em duality}'' between compact and wrapped Fukaya categories of $W$. 
We start Section \ref{section the case of compact Fukaya category} by giving a specific example satisfying the ``duality''.

Let $T$ be a tree and let $P_n(T)$ be the plumbing of the cotangent bundles of $T^*S^n$ along $T$ as in \cite{Bae-Choa-Jeong-Karabas-Lee}. For each vertex $v$ of $T$, let $S_v$ be the Lagrangian sphere in $P_n(T)$ corresponding to $v$, and let $L_v$ be the Lagrangian cocore disk corresponding to $v$. This means that the Lagrangian spheres $S_v$ and the Lagrangian cocore disks $L_v$ intersect transversely and that the intersection numbers between those are given by
\[ |S_v \cap L_w|  =\begin{cases} 1 & v= w,\\ 0 &\text{ otherwise.}  \end{cases} \]

We note that \cite{Bae-Choa-Jeong-Karabas-Lee} compares the categorical entropies on compact and wrapped Fukaya categories of $P_n(T)$ by using \cite[Lemma 2.5]{Abouzaid-Smith} and the above Lagrangians $\{S_v\}$ and $\{L_v\}$. 
Motivated by this, we will assume the following in this subsection.

\begin{assumption}
	\label{assumption1}	
 There exists a finite collection of exact, closed Lagrangians $\{S_i\}_{i \in I}$ of $\hW$ indexed by some set $I$ such that
\begin{enumerate}
	\item the direct sum $S = \oplus_{i \in I} S_v$ split-generates the compact Fukaya category $\mathcal{F}(W)$ in such way that every exact, closed Lagrangian $L$ of $\hW$ is quasi-isomorphic to a twisted complex for $L$ with components $\{S_i\}$, in which none of the arrows are nonzero multiples of the identity morphisms, and
	\item there exists another collection of Lagrangians $\{L_i\}_{i \in I}$ of $\hW$, each of which intersects $S_i \in I$ transversely and satisfies
		\[ |S_i \cap L_j|  =\begin{cases} 1 & i= j,\\ 0 &\text{ otherwise.}  \end{cases} \]
\end{enumerate}
\end{assumption}

Let $S = \bigoplus_{i \in I} S_i$ and $L = \bigoplus_{i \in I} L_i$. 
Let us denote by $\Phi_{\mathcal{F}(W)}$ the auto-functor on $\mathcal{F}(W)$ induced by $\phi$.
Then, since the arguments in \cite[Lemma 6.5, Theorem 6.6]{Bae-Choa-Jeong-Karabas-Lee} continue to work under Assumption \ref{assumption1}, we have Lemma \ref{lem compact entropy}. 

\begin{lem}
	\label{lem compact entropy}
	For any exact, compactly-supported symplectic automorphism $\phi$ on a Weinstein manifold $\hW$, if $\hW$ satisfies Assumption \ref{assumption1}, then 
	\[ h_{cat} (\Phi_{\mathcal{F}(W)}) = \lim_{n \to \infty} \frac{1}{n} \log \dim HF^*(\phi^n(S) , L) . \]
\end{lem}

\begin{thm}
	\label{thm main theorem for compact}
	Let a pair $(\hW,\phi: \hW \to \hW)$	be as in Lemma \ref{lem compact entropy}. Then the categorical entropy $h_{cat}(\Phi_{\mathcal{F}(W)})$ for its induced functor on the compact Fukaya category $\mathcal{F}(W)$ bounds the topological entropy of $\phi$ from below, i.e.,
	\[ h_{cat} (\Phi_{\mathcal{F}(W)}) \leq h_{top} (\phi).\]
\end{thm}
\begin{proof}
Basically, most arguments in the proof of Theorem \ref{thm main} can be applied to this case. Indeed, for $n \in \mathbb{N}$, we once again apply Lemma \ref{lemma Lagrangian tomograph} to the pair $( S, \phi^{-n}(L) )$ to get a family of Lagrangians $\{S^s\}_{s \in B_{\epsilon}^d}$ such that
\begin{enumerate}
	\item[(i)] $S$ and $S^s$ are Hamiltonian isotopic to each other for all $s \in B_\epsilon^d$, 
	\item[(ii)] $d_H(S, S^s) < \frac{\epsilon}{2}$, and
	\item[(iii)] $S^s \pitchfork \phi^{-n}(L)$ for almost all $s \in B_\epsilon^d$. 
\end{enumerate}
As mentioned in the proof of \ref{thm main}, one can find such a family $\{S^s\}_{s\in B_{\epsilon}^d}$, for which the third condition (iii) holds for all $n\in \mathbb{N}$.

Then we consider the following function.
\[N_n(s) := |S^s \cap \phi^{-n}(L)|.\]

By applying Lemma \ref{lemma Crofton's inequality} once again, we have 
\begin{gather}
	\label{eqn cpt proof 1}
	\int_{B^d_\epsilon} N_n(s) ds \leq C' \cdot \vol\left(\phi^{-n}(L) \cap W\right).
\end{gather}
for some constant $C'$ which does not depend on $n$.

On the other hand, for $S^s \pitchfork \phi^{-n}(L)$, we have 
\begin{gather}
	\label{eqn cpt proof 2}
	\mathrm{dim}HF \left(\phi^n(S),L\right) = \mathrm{dim}HF\left(S,\phi^{-n}(L)\right) =\mathrm{dim}HF\left(S^s, \phi^{-n}(L) \right) \leq N_n(s).
\end{gather}
The above inequality holds since $S$ and $S^s$ are Hamiltonian isotopic.

By integrating Equation \eqref{eqn cpt proof 2}, one has 
\begin{gather}
	\label{eqn cpt proof 3}
	\vol(B^d_\epsilon) \cdot \mathrm{dim}HF \left(\phi^n(S),L\right) = \int_{B^d_\epsilon} \mathrm{dim}HF \left(\phi^n(S),L\right) ds \leq \int_{B^d_\epsilon} N_n(s) ds.
\end{gather}

From two inequalities \eqref{eqn cpt proof 1} and \eqref{eqn cpt proof 3}, one has 
\begin{gather}
	\label{eqn cpt proof 4}
	\vol(B^d_\epsilon) \cdot \mathrm{dim}HF \left(\phi^n(S),L\right) \leq C' \cdot \vol\left(\phi^n(G) \cap W\right).
\end{gather}
By taking $\limsup_{n \to \infty} \frac{1}{n} \log$ for the both hand sides of \eqref{eqn cpt proof 4} and using Lemma \ref{lem compact entropy}, one has 
\begin{align}
	\label{eqn cpt proof 5}
	h_{cat} (\Phi_{\mathcal{F}(W)}) &= \limsup_{n\to\infty} \frac{1}{n} \log \dim HF \left(\phi^n(S),L\right) \\
	\notag &\leq \limsup_{n \to \infty} \frac{1}{n} \log \vol\left(\phi^{-n}(L) \cap W\right).
\end{align}

Here the latter is again bounded above by $h_{top}(\phi)$ due to Theorem \ref{thm Yomdin thm}. Therefore, \eqref{eqn cpt proof 5} proves the assertion.
\end{proof}

\section{Barcode entropy}
\label{section barcode entropy}
In this section, we define another entropy, called {\em barcode entropy}. 
As mentioned in Section \ref{subsection results}, the notion of barcode entropy is the same as the {\em relative barcode entropy} defined in \cite{Cineli-Erman-Ginzburg}. 
At the end of Section \ref{section barcode entropy}, we give further questions related to categorical, topological, and barcode entropies. 

\subsection{Preliminaries}
\label{subsection preliminaries}
In this subsection, we review the theory of persistence modules, and we apply it to Lagrangian Floer homology. 
We refer the reader to \cite{Polterovich-Rosen-Samvelyan-Zhang,Usher-Zhang} for the theory of persistence modules. 
Also, we refer the reader to \cite{Cineli-Erman-Ginzburg} for the details we omitted in the current subsection.

The notion of {\em non-Archimedean norm} on a vector space is defined in \cite[Definition 2.2]{Usher-Zhang}. 
It is easy to check that $\mathcal{A}$ in \eqref{eqn non-archimedean norm} is a {\em non-Archimedean norm} on the $k$-vector space $CF(L_1,L_2)$.
Moreover, $CF(L_1,L_2)$ is {\em orthogonal} with respect to $\mathcal{A}$. 

Now, we are ready to apply \cite[Theorem 3.4]{Usher-Zhang} for the differential
\[\delta: CF(L_1,L_2) \to CF(L_1,L_2),\]
Since $\delta$ is a linear self-mapping of an orthogonal vector space $CF(L_1,L_2)$, one obtains a basis $\Sigma = \{\alpha_i, \beta_j, \gamma_j\}$ of $CF(L_1,L_2)$ satisfying 
\begin{enumerate}
	\item $\partial \alpha_i = 0$, 
	\item $\partial \gamma_j = \beta_j$, and
	\item $\mathcal{A}(\gamma_1) - \mathcal{A}(\beta_1) \leq \mathcal{A}(\gamma_2) - \mathcal{A}(\beta_2) \leq \dots$.
\end{enumerate}

By using the above, we define the followings.
\begin{definition}
	\label{def barcode}
	\mbox{}
	\begin{enumerate}
		\item A {\em bar} of $CF(L_1,L_2)$ is either $\alpha_i$ or a pair $(\beta_j, \gamma_j)$. 
		\item The {\em length of a bar $b$} is given by 
		\begin{equation*}
			\text{the length of  } b = 
			\begin{cases}
				\infty & \text{if  } b = \alpha_i, \\
				\mathcal{A}(\gamma_j) - \mathcal{A}(\beta_j) & \text{otherwise}.
			\end{cases} 
		\end{equation*}
		\item Let $b_\epsilon(L_1,L_2)$ be the number of bars of $CF(L_1,L_2)$ whose lengths are greater than or equal to $\epsilon$. 
	\end{enumerate}
\end{definition}

\begin{remark}
	\label{rmk length of bars}
	We note that $\mathcal{A}$ is unique up to constant.
	More precisely, for an exact Lagrangian $L_i$, a choice of primitive function $h_i: L_i \to \mathbb{R}$ is not unique, but unique up to constant.  
	Thus, it is easy to show that the length of bars depends only on $L_i$ and independent of the choice of primitive function $h_i: L_i \to \mathbb{R}$. 
\end{remark}

By Definition \ref{def barcode}, Lemma \ref{lemma inequality 1} is obvious. 
\begin{lem}
	\label{lemma inequality 1}
	Let $L_1$ and $L_2$ be a transversal pair of Lagrangians. Then, for any $\epsilon \geq 0$,
	\[b_\epsilon(L_1,L_2) \leq b_0(L_1, L_2) \leq |L_1 \cap L_2|.\]
\end{lem}

It is well-known that $b_\epsilon$ is insensitive to small perturbations of the Lagrangians with respect to the Hofer distance. 
More precisely, Lemma \ref{lemma inequality 2} holds.
\begin{lem}
	\label{lemma inequality 2}
	Let $L_1'$ be a Lagrangian satisfying 
	\begin{itemize}
		\item $L_1'$ and $L_1$ are Hamiltonian isotopic to each other, 
		\item $d_H(L_1, L_1') < \frac{\delta}{2}$ with $\delta < \epsilon$, and
		\item $L_1'$ and $L_2$ are transversal to each other. 
	\end{itemize}
	Then, 
	\[b_{\epsilon+\delta}(L_1',L_2) \leq b_\epsilon(L_1,L_2) \leq b_{\epsilon - \delta}(L_1',L_2).\]
\end{lem}
\begin{proof}
	See \cite[Equations (3,13) and (3,14)]{Cineli-Erman-Ginzburg}. 
	We also refer the reader to \cite{Kislev-Shelukhin, Polterovich-Rosen-Samvelyan-Zhang, Usher-Zhang}. 
\end{proof}

Now, we extend the barcode counting function $b_\epsilon(L_1,L_2)$ to a {\em good pair $(L_1,L_2)$} defined in Definition \ref{def good pair}.
For a good pair, we set 
\[b_\epsilon(L_1, L_2) := \liminf_{d_H(L_2, L_2') \to 0} b_\epsilon(L_1, L_2'),\]
where the limit is taken over Lagrangians $L_2'$ such that $L_2' \pitchfork L_1$.
We also note that $L_2$ and $L_2'$ should be Hamiltonian isotopic so that the Hofer distance between them is defined. 

\subsection{Barcode entropy}
\label{subsection barcode entropy}
In the rest of this paper, we consider the same situation as what we considered in Section \ref{section categorical vs topological entropy}. 
For the reader's convenience, we review the setting.

Let $(\hW, d\lambda)$ be a Weinstein manifold and let $\phi: \hW \to \hW$ be a compactly supported exact symplectic automorphism. 
Then, there is a Weinstein domain $W$ such that 
\begin{itemize}
	\item $\hW = W \cup \left( [1,\infty) \times \partial W\right)$, 
	\item $\lambda|_{[1,\infty \times \partial W)} = r \alpha$ where $r$ is a coordinate for $[1.\infty)$ and $\alpha := \lambda|_{\partial W}$, and 
	\item the support of $\phi$ is contained in $\mathrm{Int}(W)$. 
\end{itemize}

Let $(L_1,L_2)$ be a good pair of Lagrangians with respect to $W$, i.e., 
\[L_1 \cap L_2 \cap \left([1,\infty) \times \partial W \right) = \varnothing.\]
Since $\phi$ is the identity outside of $W$, $\left(L_1,\phi^n(L_2)\right)$ is also a good pair for any $n \in \mathbb{Z}$.
Then, for a fixed $\epsilon$, $b_\epsilon\left(L_1, \phi^n(L_2)\right)$ is well-defined.  
We would like to define the {\em barcode entropy of $\phi$} as the exponential growth rate of $b_\epsilon\left(L_1, \phi^n(L_2)\right)$ as $n \to \infty$.

To be more precise, let $\log^+: \mathbb{Z}_{\geq 0} \to \mathbb{R}$ be the function defined as 
\[\log^+(k) = \begin{cases}
	0 \text{  if  } k =0, \\
	\log(k)  \text{  other wise,  }
\end{cases}\]
where the logarithm is taken base $2$.

\begin{definition}
	\label{def barcode entropy}
	\mbox{}
	\begin{enumerate}
		\item For any $\epsilon \in \mathbb{R}_{\geq 0}$, the {\em $\epsilon$-barcode entropy of $\phi$ relative to $(L_1,L_2)$} is 
		\[h_\epsilon(\phi;L_1,L_2) := \lim_{n \to \infty} \frac{1}{n} \log^+ b_\epsilon\left(L_1, \phi^n(L_2)\right).\]
		\item The {\em barcode entropy of $\phi$ relative to $(L_1,L_2)$} is 
		\[h_{bar}(\phi;L_1,L_2) := \lim_{\epsilon \searrow 0} h_\epsilon(\phi;L_1,L_2).\]
	\end{enumerate}
\end{definition}
We note that Definition \ref{def barcode entropy} is the same as the notion of {\em relative barcode entropy} in \cite{Cineli-Erman-Ginzburg}, except a minor adjustment to our set up. 

\subsection{Barcode vs topological entropy}
\label{subsection barcode vs topological entropy}
In this subsection, we prove that for any good pair $(L_1, L_2)$, the barcode entropy of $\phi$ bounds the topological entropy of $\phi$ from below. 
The proof of Proposition \ref{prop bar vs top} is almost same as the \cite[Proof of Theorem A]{Cineli-Erman-Ginzburg}.

\begin{prop}[= The second inequality in Proposition \ref{prop barcode}]
	\label{prop bar vs top}
	For any good pair $(L_1,L_2)$, 
	\[h_{bar}(\phi;L_1,L_2) \leq h_{top}(\phi).\]
\end{prop}
\begin{proof}
	If $h_{bar}(\phi;L_1,L_2) =0$, then there is nothing to prove.
	Thus, let assume that $h_{bar}(\phi;L_1,L_2) >0$. 
	We would like to show that if $\alpha\leq h_{bar}(\phi;L_1,L_2)$, then $\alpha \leq h_{top}(\phi)$. 
	
	Let $\delta$ be a positive number.
	Since 
	\[h_{bar}(\phi;L_1,L_2) := \lim_{\epsilon \searrow 0} h_\epsilon(\phi;L_1,L_2) \geq \alpha,\]
	there is $\epsilon_0>0$ such that if $\epsilon < \epsilon_0$, then $h_\epsilon(\phi;L_1,L_2) > \alpha -\delta$.
	We fix a positive number $\epsilon$ such that $2\epsilon < \epsilon_0$. 
	
	Since 
	\[h_{2 \epsilon}(\phi;L_1,L_2) = \limsup_{n \to \infty}\frac{1}{n} \log^+ b_{2 \epsilon}\left(L_1,\phi^n(L_2)\right) > \alpha - \delta,\]
	there is an increasing sequence of natural numbers $\{n_i\}_{i \in \mathbb{N}}$ such that 
	\begin{gather}
		\label{eqn inequality 3}
		b_{2\epsilon}\left(L_1,\phi^{n_i}(L_2)\right) > 2^{(\alpha-\delta)n_i}.
	\end{gather}
	
	Now, we apply Lemma \ref{lemma Lagrangian tomograph} to the good pair $\left(L_1, \phi^{n_i}(L_2)\right)$.
	Then, one obtains a family of Lagrangians $\{L^s\}_{s \in B^d_{\epsilon, n_i}}$ such that 
	\begin{enumerate}
		\item[(i)] $L_1$ and $L^s$ are Hamiltonian isotopic for all $s \in B^d_{\epsilon, n_i}$, 
		\item[(ii)] $d_H(L_1,L^s) < \frac{\epsilon}{2}$ for all $s \in B^d_{\epsilon, n_i}$, and 
		\item[(iii)] $L^s \pitchfork \phi^{n_i}(L_2)$ for almost all $s \in B^d_{\epsilon, n_i}$.
	\end{enumerate}

	We point out that by the argument in the proof of Theorem \ref{thm main}, we can choose a family satisfying (i)--(iii) for all $n_i$. 
	Let $\{L^s\}_{s \in B^d_\epsilon}$ denote a fixed Lagrangian tomograph. 
%
%
%
	
	Let 
	\[N_i(s):= |L^s \cap \phi^{n_i}(L_2)|.\]
	Then, one has 
	\[\int_{B^d_\epsilon} N_i(s) ds \leq C \cdot \vol(\phi^{n_i}(L_2) \cap W),\]
	by applying Lemma \ref{lemma Crofton's inequality}.
	We note that $C$ is a constant independent from $n_i$. 
	
	From Lemmas \ref{lemma inequality 1} and \ref{lemma inequality 2} and Equation \eqref{eqn inequality 3}, we have 
	\[2^{(\alpha-\delta)n_i} \leq b_{2\epsilon}\left(L_1,\phi^{n_i}(L_2)\right) \leq b_\epsilon\left(L^s,\phi^{n_i}(L_2)\right) \leq |L^s \cap \phi^{n_i}(L_2)| = N_i(s).\]
	Then, by taking integration over $B^d_{\epsilon}$, one has 
	\[\vol(B^d_\epsilon) \cdot 2^{(\alpha-\delta)n_i} \leq C \cdot \vol\left(\phi^{n_i}(L_2) \cap W\right).\]
	Since $\vol(B^d_\epsilon)$ and $C$ do not depend on $n_i$, 
	\begin{gather}
		\label{eqn final inequality}
		\alpha - \delta \leq \limsup_{i\to \infty} \frac{1}{n_i} \log^+ \vol\left(\phi^{n_i}(L_2) \cap W \right) \leq h_{top}(\phi|_W).
	\end{gather}
	The last inequality holds due to Theorem \ref{thm Yomdin thm}, and because of the fact that 
	\[\phi^{n_i}(L_2) \cap W = \phi^{n_i}(L_2 \cap W) = (\phi|_W)^{n_i}(L_2 \cap W).\]
	
	Finally, we note that $\phi$ is compactly supported, and that $\mathrm{supp}(\phi) \subset W$. 
	Thus, 
	\[h_{top}(\phi) = h_{top}(\phi|_W).\]

	Thus, one has 
	\[\alpha - \delta \leq h_{top}(\phi).\] 
	This completes the proof.
\end{proof}

\subsection{Barcode vs categorical entropy}
\label{subsection barcode vs categorical entropy}
In the previous section, for an arbitrary good pair $(L_1,L_2)$, we compared the barcode entropy of a triple $(\phi;L_1,L_2)$ and the topological entropy of $\phi$.
As the result, we proved Proposition \ref{prop bar vs top}.
In this subsection, we compare barcode and categorical entropy.
However, in order to compare them, we should choose some specific pairs of Lagrangians. 

First, we choose a stop $\Lambda$ giving a fully stopped partially wrapped Fukaya category $\cW(W, \Lambda)$.
Let $G$ be an embedded Lagrangian generating $\cW(W,\Lambda)$. 
As we did in the proof of Theorem \ref{thm main}, let $\varphi_0$ be a Hamiltonian isotopy satisfying
\begin{enumerate}
	\item[(A)] $\left(\varphi_0(G),\phi^n(G)\right)$ is a good pair for all $n \in \mathbb{N}$, 
	\item[(B)] $HW_\Lambda\left(G,\phi^n(G)\right) = HF\left(\varphi_0(G),\phi^n(G)\right)$ for all $n \in \mathbb{N}$.
\end{enumerate}

As we showed before, one has 
\[h_{cat}(\Phi) = h_{cat}(\Phi_\Lambda) = \limsup_{n\to\infty} \frac{1}{n} \log^+ \dim HF\left(\varphi_0(G),\phi^n(G)\right).\]
Since $\dim HF\left(\varphi_0(G),\phi^n(G)\right)$ equals the number of bars having infinite length, one has 
\[\dim HF\left(\varphi_0(G),\phi^n(G)\right) \leq b_\epsilon\left(\varphi_0(G),\phi^n(G)\right).\]
This induces Proposition \ref{prop bar vs cat}.

\begin{prop}[= The first inequality in Proposition \ref{prop barcode}]
	\label{prop bar vs cat}
	For $G$ and $\varphi_0$ given above, 
	\[h_{cat}(\Phi) \leq h_{bar}\left(\phi; \varphi_0(G), G\right).\]
\end{prop}

\begin{remark}
	\label{rmk zero barcode}
	We note that there always exists a good pair $(L_1,L_2)$ such that $h_{bar}(\phi;L_1,L_2)=0$. 
	By choosing a Lagrangian $L_2$ such that $L_2$ does not intersect the support of $\phi$, one obtains a such pair. 
	Thus, the choice of $G$ (and $\varphi_0$) in Proposition \ref{prop bar vs cat} is essential.
\end{remark}

\subsection{Further questions} 
\label{subsection further questions 2}
In this subsection, we discuss the questions given in Section \ref{subsection further questions} in more detail. 

We also recall that 
\begin{gather*}
	h_{top}(\phi) \geq \text{  the exponential growth rate of  } \vol\left(\phi^n(Y)\right),
\end{gather*}
for any compact submanifold $Y$ by \cite{Newhouse, Przytycki}.
And it is known by \cite{Yomdin} that
\begin{gather}
	\label{eqn topological entropy 2}
	h_{top}(\phi) = \sup_{\text{compact submanifold  } Y \subset W} \left(\text{  the exponential growth rate of  } \vol\left(\phi^n(Y)\right)\right).
\end{gather}

As one can see in the proof of Proposition \ref{prop bar vs top}, $h_{bar}(\phi;L_1, L_2)$ bounds the exponential volume growth rate of $\phi^n(L_2)$ from below.
Thus,  
\[h_{top}(\phi) \geq h_{bar}(\phi;L_1, L_2).\]
As a generalization of Equation \eqref{eqn topological entropy 2}, one can ask whether the following equality holds or not:
\[h_{top}(\phi) = \sup_{(L_1, L_2) \text{  is a good pair}} h_{bar}(\phi;L_1,L_2).\]

The supremum in the above equation runs over the set of all good pairs. 
As mentioned in Remark \ref{rmk zero barcode}, it is easy to find a good pair $(L_1,L_2)$ such that $h_{bar}(\phi;L_1,L_2)=0$. 
Thus, we would like to remove such good pairs from the set where the supremum runs over, for computational convenience.

Finally, we ask whether the following equality holds or not:
\[h_{top}(\phi) = \sup_{G, \varphi_0} h_{bar}\left(\phi;\varphi_0(G),G)\right),\]
where $G$ is a generating Lagrangian and $\varphi_0$ is a Hamiltonian isotopy satisfying the conditions in Section \ref{subsection barcode vs categorical entropy}.

On the other hand, one can ask a similar question for $h_{cat}(\phi)$.
More precisely, we ask whether the following equality holds or not:
\[h_{cat}(\phi) = \inf_{G, \varphi_0} h_{bar}\left(\phi;\varphi_0(G),G)\right).\]

\bibliographystyle{amsalpha}
\bibliography{entropy2}
\end{document}